\documentclass{amsart}
\usepackage{amsmath,amssymb}
\usepackage{graphicx,subfigure}

\usepackage{color}

\newtheorem{theorem}{Theorem}[section]
\newtheorem{proposition}[theorem]{Proposition}

\newtheorem{lemma}[theorem]{Lemma}

\theoremstyle{definition}
 \newtheorem{definition}[theorem]{Definition}

\newtheorem*{definition*}{Definition}

\theoremstyle{remark}
\newtheorem{remark}[theorem]{Remark}

\newcommand\ssubset{\subset\!\subset}
\newcommand\ssupset{\supset\!\supset}

\numberwithin{equation}{section}

\renewcommand\L{\langle l\rangle}

\newcommand{\al}{\alpha}
\newcommand{\be}{\beta}
\newcommand{\de}{\delta}
\newcommand{\ep}{\varepsilon}

\newcommand{\la}{\lambda}
\newcommand{\om}{\omega}
\newcommand{\si}{\sigma}
\newcommand{\te}{\theta}
\newcommand{\vp}{\varphi}

\newcommand{\De}{\Delta}
\newcommand{\Ga}{\Gamma}
\newcommand{\La}{\Lambda}
\newcommand{\Si}{\Sigma}
\newcommand{\Om}{\Omega}

\newcommand{\tu}{\widetilde{u}}

\newcommand{\tN}{\widetilde{N}}

\def\hv{\widehat v}
\newcommand{\hu}{\widehat{u}}
\def\hpsi{\widehat \psi}

\def\RR{\mathbb{R}}

\def\ZZ{\mathbb{Z}}

\def\TT{\mathbb{T}}

\renewcommand\SS{\mathbb{S}}
\newcommand{\cB}{{\mathcal B}}
\newcommand{\cC}{{\mathcal C}}

\newcommand{\cE}{{\mathcal E}}

\newcommand{\cI}{{\mathcal I}}

\newcommand{\cK}{{\mathcal K}}

\newcommand{\cM}{{\mathcal M}}

\newcommand{\cP}{{\mathcal P}}

\newcommand{\cS}{{\mathcal S}}
\newcommand{\cT}{{\mathcal T}}
\newcommand{\cX}{{\mathcal X}}

\newcommand{\cV}{{\mathcal V}}

\newcommand{\pd}{\partial}
\newcommand\minus\backslash
\newcommand{\id}{{\rm id}}

\newcommand\lan\langle
\newcommand\ran\rangle

\newcommand{\rank}{\operatorname{rank}}
\newcommand{\supp}{\operatorname{supp}}
\newcommand{\Span}{\operatorname{span}}

\DeclareMathOperator\Real{Re}

\DeclareMathOperator\dist{dist}

\renewcommand\leq\leqslant
\renewcommand\geq\geqslant
\newlength{\intwidth}

\newcommand\loc{_{\mathrm{loc}}}

\DeclareMathOperator\Imag{Im}

\newcommand\co{^{\mathrm{c}}}
\def\sq{{\tau_-^{1/2}}}
\def\sqP{{\tau_+^{1/2}}}
\def\sqp{{\tau^{1/2}}}
\def\sqa{{|\tau|^{1/2}}}
\def\cKp{\cK^\perp}

\newcommand\step[1]{{\noindent{\em #1.}}}

\newcommand{\triplei}[1]{{\left\vert\kern-0.25ex\left\vert\kern-0.25ex\left\vert #1 
        \right\vert\kern-0.25ex\right\vert\kern-0.25ex\right\vert_{\tau,\infty}}}

\newcommand{\triple}[1]{{\left\vert\kern-0.25ex\left\vert\kern-0.25ex\left\vert #1 
    \right\vert\kern-0.25ex\right\vert\kern-0.25ex\right\vert_{\tau,2}}}

\begin{document}

\title[Approximation theorems for the Schr\"ondiger equation]{Approximation
  theorems for the Schr\"odinger equation and quantum vortex reconnection}

\author{Alberto Enciso}
\address{Instituto de Ciencias Matem\'aticas, Consejo Superior de
  Investigaciones Cient\'\i ficas, 28049 Madrid, Spain}
\email{aenciso@icmat.es}

\author{Daniel Peralta-Salas}
\address{Instituto de Ciencias Matem\'aticas, Consejo Superior de
 Investigaciones Cient\'\i ficas, 28049 Madrid, Spain}
\email{dperalta@icmat.es}

%
%
\begin{abstract}
  We prove the existence of smooth solutions to the Gross--Pitaevskii
  equation on~$\RR^3$ that feature arbitrarily complex quantum vortex
  reconnections. We can track the evolution of the vortices during the
  whole process. This permits to describe the reconnection events
  in detail and verify that this scenario exhibits the properties
  observed in experiments and numerics, such as the $t^{1/2}$ and
  change of parity laws. We are mostly interested in solutions tending
  to~1 at infinity, which have finite Ginzburg--Landau energy and
  physically correspond to the presence of a background chemical
  potential, but we also consider the cases of Schwartz initial data
  and of the Gross--Pitaevskii equation on the torus. An essential
  ingredient in the proofs is the development of novel global
  approximation theorems for the Schr\"odinger equation
  on~$\RR^n$. Specifically, we prove a qualitative approximation
  result that applies for solutions defined on very general spacetime
  sets and also a quantitative result for solutions on product
  sets in spacetime $D\times\RR$. This hinges on
  frequency-dependent estimates for the Helmholtz--Yukawa equation that
  are of independent interest.
\end{abstract}
\maketitle

\section{Introduction}

The Gross--Pitaevskii equation,
\begin{equation}\label{GP}
i\pd_tu+\De u+ (1-|u|^2) u=0\,,\qquad x\in\RR^3\,,
\end{equation}
models the evolution of a Bose--Einstein condensate (sometimes called
a superfluid). This is an important instance of a nonlinear
Schr\"odinger equation, which has the
peculiarity that, instead of looking for solutions that decay at
infinity, one is often interested in functions that tend to~1 as
$|x|\to\infty$. From a physical point of view, this is related to the
consideration of a chemical potential at infinity; mathematically,
one can relate the Gross--Pitaevskii equation with the Ginzburg--Landau functional
\[
\cE[u](t):= \int_{\RR^3}\bigg(\frac12|\nabla u(x,t)|^2 +
\frac14\big(1-|u(x,t)|^2\big)^2\bigg)\, dx\,,
\]
so one picks solutions that tend to~1 fast enough at infinity to have
finite Ginzburg--Landau energy. 

\subsection{Reconnection of quantum vortices for the Gross--Pitaevskii
  equation}

A hot topic in condensed matter physics is the study of the evolution
of quantum vortices~\cite{Aranson}. Recall that the {\em quantum vortices}\/ of the superfluid at
time~$t$ are defined as the connected components of the set
\[
Z_u(t):=\{x\in\RR^3: u(x,t)=0\}\,,
\]
so, as $u$ is complex valued, they are typically given by closed curves in
space. A central aspect is the analysis of the {\em vortex
  reconnection}\/, that is, the process through which two quantum
vortices cross, each of them breaking into two parts and exchanging part
of itself for part of the other (see Figure~\ref{Figure}, top). This may lead to a change of topology
of the quantum vortices. Among the extensive
literature on this topic, an outstanding contribution is the first
experimental measurement of vortex reconnection in superfluid
helium~\cite{Bewley}. In this paper it was observed that the distance
between the vortices behaves as $C|t-T|^{1/2}$ near the reconnection
time~$T$. Further numerical~\cite{Irvine,Villois} and theoretical~\cite{Nazarenko} studies have analyzed quantum vortex
reconnections (of very different global properties) in
detail, showing that the above separation rate
is in fact universal (this is nowadays called the {\em
  $t^{1/2}$~law}\/). Another intriguing  numerical
observation~\cite{Irvine} is that the parity of the number of quantum
vortices changes at reconnection time, meaning that an even number of
vortices reconnect into an odd number of quantum vortices and
viceversa.

\begin{figure}[t]
  \centering
  \subfigure{
\includegraphics[scale=1,angle=0]{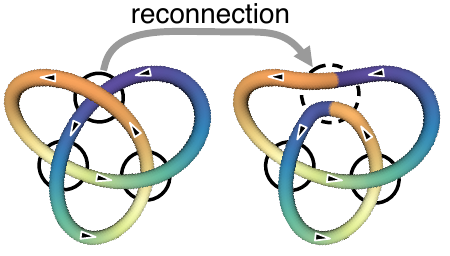}}\\[2mm]
  \subfigure{
\includegraphics[scale=0.6]{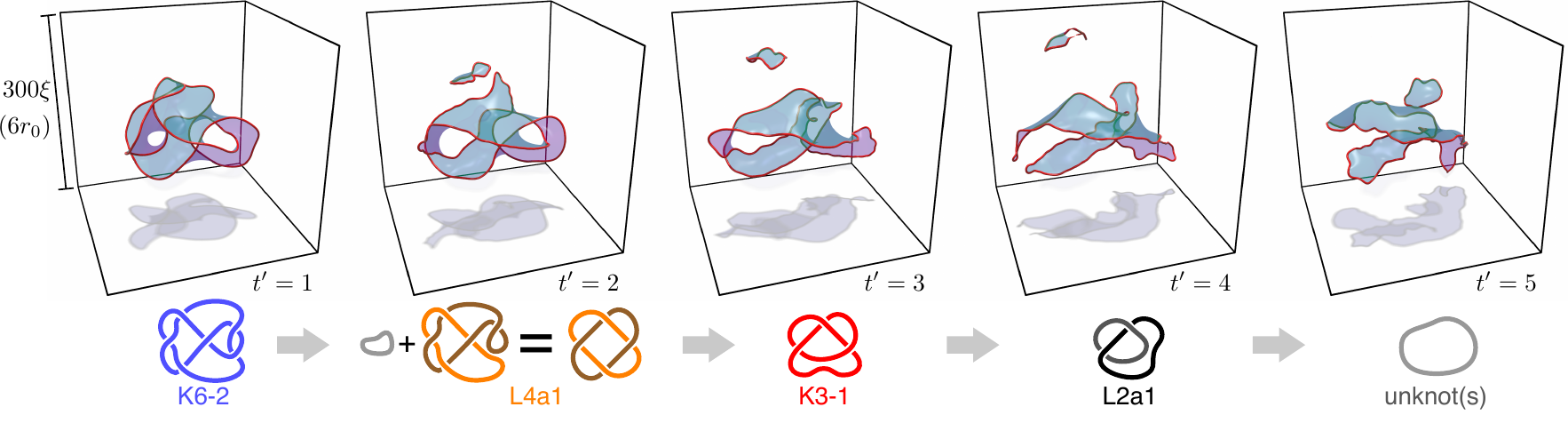}}
\caption{{\em Top:} Visual description of the reconnection
  phenomenon. Here a quantum vortex in the shape of a trefoil knot
  reconnects into two linked unknots. {\em Bottom:} Numerical
  simulation of a solution to the Gross--Pitaevskii equation that
  exhibits a cascade of reconnections that transform a K6-2 knot into
  an unknot. Courtesy of  Irvine, Kauffman and Kleckner~\cite{Irvine}.} \label{Figure}
\end{figure}

As an aside, let us recall that the Gross--Pitaevskii equation, and
other nonlinear Schr\"odinger equations, are somehow connected with
the 3D Euler equation~\cite{Banica2, Banica}. This provides some heuristic
relation between the quantum vortices of a Bose--Einstein condensate
and vortex filaments in an incompressible
fluid~\cite{Jerrard1,Jerrard3, Kenig,Lannes}. However, in this paper we will not
pursue this line of ideas.

Our motivation for this paper is to prove the reconnection of quantum
vortices in smooth solutions to the Gross--Pitaevskii equation. More
precisely, in view of the experimental and numerical evidence, there
are two issues that we want to analyze in this context. Firstly, we
aim to show that, just as in the physics literature, in the
reconnections we construct the distance between vortices near the
reconnection time obeys the $t^{1/2}$~law. Secondly, we aim to track
the vortex reconnection process at all times, both locally and
globally, even if the topology of the initial and final vortices are
completely different. This is motivated by the numerical
evidence~\cite{Irvine} that, when looked at from a global point of
view, vortex reconnection can occur so that the topology (i.e., the knot
and link type) of the vortices change wildly.

Our main result shows that, given any finite initial and
final configurations of quantum vortices (which do not need to be
topologically equivalent) and any conceivable way of reconnecting them
(that is, of transforming one into the other), there is a smooth
initial datum~$u_0$ whose associated solution realizes this specific
vortex reconnection scenario. 

To make this statement precise, one can describe the initial and final
vortex configurations by links~$\Ga_0, \Ga_1\subset\RR^3$. By a {\em link}\/ we
denote a finite union of closed pairwise disjoint curves without
self-intersections, contained in~$\RR^3$, and of class~$C^\infty$.
Notice that $\Ga_0$ and~$\Ga_1$ do not necessarily have the same
number of connected components, and that these components need not be
homeomorphic. To describe a way of transforming the link~$\Ga_0$
into~$\Ga_1$ in time~$T$, we introduce the notion of a {\em
  pseudo-Seifert surface}\/. By this we will mean a smooth,
two-dimensional, bounded, orientable surface $\Si\subset\RR^4$ whose
boundary is
  \[
\pd\Si=(\Ga_0\times\{0\})\cup (\Ga_1\times\{T\})\,.
\]
As an additional technical assumption, we will assume that the surface
is in generic position, meaning that the fourth (``time'') coordinate
of~$\RR^4$ is a Morse function on~$\Si$ that does not have any
critical points on the boundary~$\pd\Si$. This kind of pseudo-Seifert
surfaces can be used to describe any reconnection cascade like the
ones numerically studied in~\cite{Irvine} (see Figure~\ref{Figure},
    bottom for an illustrative example). As a matter of fact, we show in Section~\ref{S.observed}
that pseudo-Seifert surfaces provide a universal mechanism of describing the reconnection
process for the initial and final links~$\Ga_0,\Ga_1$.

The theorem can then be stated as follows. To state this result, let us begin by
introducing some notation. Given a spacetime subset
$\Om\subset\RR^{n+1}$ (here, $n=3$), let us denote
by 
\begin{equation}\label{Omt}
\Om_t:=\Om\cap (\RR^n\times\{t\})
\end{equation}
its intersection with the time~$t$ slice. Furthermore, we use the notation
\begin{equation}\label{dilation}
\La_\eta(x):= \eta x
\end{equation}
for the dilation on~$\RR^3$ with ratio~$\eta>0$.

\begin{theorem}\label{T.GP}
  Consider two links $\Ga_0,\Ga_1\subset\RR^3$ and a pseudo-Seifert
  surface~$\Si\subset\RR^4$
  connecting~$\Ga_0$ and~$\Ga_1$ in time~$T>0$. Then, there is a global smooth solution~$u(x,t)$ to the Gross--Pitaevskii
  equation on~$\RR^{3}$, tending to~$1$ at infinity, which 
  realizes the vortex reconnection pattern
  described by~$\Si$ up to a diffeomorphism. Specifically, for any $\ep>0$ and
  any~$k>0$, one has:
  \begin{enumerate}
    \item The function~$u$ tends to~$1$ exponentially fast at infinity.
      More precisely, $1-u\in C^\infty\loc(\RR, \cS(\RR^3))$, where
      $\cS(\RR^3)$ is the Schwartz space.
    \item One can track the evolution of the quantum vortices during
      the prescribed reconnection process at all times. More
      precisely,  there is some $\eta>0$ and
  a diffeomorphism $\Psi$ of~$\RR^4$ with
  $\|\Psi-\id\|_{C^k(\RR^4)}<\ep$ such that $\La_{\eta}[\Psi(\Si)_t]$ is a union
  of connected components of $Z_u(\eta^2 t)$ for all~$t\in [0,T]$.
  
\item In particular, there is a smooth one-parameter family of
  diffeomorphisms $\{\Phi^t\}_{t\in\RR}$ of~$\RR^3$ with
  $\|\Phi^t-\id\|_{C^k(\RR^3)}<\ep$ and a finite
  union of closed
  intervals~$\cI\subset (0,T)$ of total length less than~$\ep$ such that $\La_\eta[\Phi^t(\Si_t)]$
  is a union of connected components of the set $Z_u(\eta^2 t)$ for all $t\in
  [0,T]\backslash\cI$.
  
\item The separation distance obeys the $t^{1/2}$ law and the parity
  of the number of quantum vortices of~$\Phi^t(\Si_t)$ changes at each
  reconnection time, in the sense
  described above.
  \end{enumerate}
  
\end{theorem}

Before presenting the main ideas of the proof of this theorem, it is
worth comparing it with our previous result with Luc\`a on vortex reconnection
for the 3D Navier--Stokes equation~\cite{NS}. From the point of view
of what we prove, the main difference is that the Navier--Stokes
result shows that one can take a finite number of ``observation
times'' $T_0<T_1<\dots< T_N$ such that the vortex structures present
at the fluid at time~$T_k$ are not topologically equivalent to those
at time $T_{k\pm1}$, which shows in an indirect way that at least one
reconnection event must have taken place. In constrast, in the above
theorem one can control the evolution of the quantum vortices during
the whole reconnection process, and in particular one can describe in
detail how the reconnection occurs. This is key to verify that these
reconnection scenarios possess the properties that are observed in the
physics literature, such as the aforementioned $t^{1/2}$ and change of
parity laws. We discuss in Section~\ref{S.observed} other relevant
physical properties that are also featured. 

From the point of view of the strategy of the proof, the result about the
Navier--Stokes equation involves two ideas. Firstly, one comes up with
a (rather sophisticated) construction of a
family of Beltrami fields (that is, eigenfunctions of the curl
operator) of arbitrarily high frequency that present vortex lines of
``robustly distinct'' topologies. This step is time-independent. The time-dependent
part of the argument hinges on the idea of transition to lower
frequencies: acting in the linear regime, the diffusive part of
the equation guarantees that a high-frequency Beltrami field can
represent the leading part of the solution at time~$T_0$ while a Beltrami
field of a still high but much lower frequency may dominate at a fixed
later time~$T_1$.

It is obvious that this heat-equation-type argument will not
work for the Gross--Pitaevskii equation even in the linear regime,
which is controlled by the Schr\"odinger equation. Our strategy is
completely different. Still, from an analytic point of view, an important simplification is that
the rescalings with parameter~$\eta$ that appear in the statement
enable us to construct solutions that tend to~1 as $|x|\to\infty$ for
which, for practical pursposes, the Gross--Pitaevskii equation
operates in a linear regime. This paves the way to using, in an essential
part of the argument, a remarkable global approximation property of
the linear Schr\"odinger equation
\begin{equation}\label{Sch}
i\pd_t v+\De v=0\,,
\end{equation}
with $x\in\RR^n$ and $n\geq2$, which to the best of our knowledge has never been observed
before.

\subsection{Global approximation theorems for the Schr\"odinger equation}

Roughly speaking, this property ensures that a function that
satisfies the Schr\"odinger equation on a spacetime set with
certain mild topological properties can be approximated, in a suitable
norm, by a global solution of the form $e^{it\De} u_0$, with $u_0$~a
Schwartz function.

All the spacetime sets we take in
this paper are assumed to have a smooth boundary unless otherwise stated. Furthermore, we will use the notation
\begin{equation}\label{L2Hs}
\|v\|_{L^2 H^s(\Om)}^2:=\int_{-\infty}^\infty \| v(\cdot,t)\|_{H^s(\Om_t)}^2\,
dt<\infty\,.
\end{equation}
A non-quantitative global approximation theorem can then be stated as follows, where
the relation $\Om'\subset\!\subset\Om$ means that the closure
of the set~$\Om'$ is contained in~$\Om$.

\begin{theorem}\label{T.GAT}
Let $v$ satisfy the Schr\"odinger equation~\eqref{Sch}
in a bounded open set with smooth boundary $\Om \subset\RR^{n+1}$ and take a smaller set
$\Om'\subset\!\subset\Om$. Suppose that $v\in L^2H^s(\Om)$ for some $s\in\RR$
and that the set $\RR^n\backslash \Om_t$ is connected for all $t\in\RR$. 
Then, for any~$\ep>0$, there is a Schwartz function $w_0\in \cS(\RR^n)$ such that $w:= e^{it\De} w_0$ approximates~$v$ as
\[
\|v-w\|_{L^2 H^s(\Om')}\leq \ep\,.
\]
\end{theorem}

\begin{remark}\label{R.reg}
As will be clear from the proof, the choice of the norm $L^2H^s(\Om)$ in
Theorem~\ref{T.GAT} is completely inessential. Instead, we could have taken
$H^s(\Om)$ for any real~$s$ or the H\"older norm~$C^s(\Om)$ for any
$s\geq0$, for instance.
\end{remark}

If the set~$\Om$ where the ``local solution''~$v$ is defined is
of the form
\begin{equation}\label{Omcyl}
  \Om=D\times\RR\,,
\end{equation}
where $D\subset\RR^n$ is a bounded open set with smooth boundary, the above
qualitative approximation result can be promoted to a quantitative statement. The local solution~$v$ must
additionally satisfy a certain decay
condition for large times (which is obviously satisfied in nontrivial examples). In order to state the quantitative result in a convenient
form, here and in what follows we denote by
\[
\hv(x,\tau):= \frac1{2\pi}\int_{-\infty}^\infty e^{-i\tau t}\, v(x,t)\, dt
\]
the time Fourier transform of a function (or tempered distribution)
$v(x,t)$ defined on~$\Om$. Also, we use the Japanese bracket
\[
\langle x\rangle:=(1+|x|^2)^{1/2}\,.
\]

\begin{theorem}\label{T.qGAT}
  Let $\Om:= D\times \RR$, where $D\subset\RR^n$ is a bounded open set with
  smooth boundary whose complement $\RR^n\backslash D$ is connected. Suppose that
  $v\in L^2(\Om)$ satisfies the Schr\"odinger equation~\eqref{Sch}
  in~$\Om$ and its time Fourier transform is bounded as
  \[
\int_{|\tau|>\tau_0}\int_D |\hv(x,\tau)|^2\, dx\, d\tau\leq
M^2\langle\tau_0\rangle^{-\si}
  \]
  for some $\si>0$ and all $\tau_0\geq0$. Then, for each $\ep\in
  (0,1)$ and any $T>0$:
  \begin{enumerate}
  \item There is a Schwartz initial datum $w_0\in \cS(\RR^n)$ such that the
    solution to the Schr\"odinger equation $w:= e^{it\De} w_0$
    approximates~$v$ on~$\Om_T:=D\times(-T,T)$ as
    \[
\|v-w\|_{L^2(\Om_T)}\leq \ep M
    \]
    and $w_0$ is bounded as
    \[
\|w_0\|_{L^2(\RR^n)}\leq e^{e^{e^{e^{C\ep^{-1/\si}}}}} M
\,. 
\]
    
    \item Given any smaller set $D'\ssubset D$, one can take an
      initial datum $w_0\in \cS(\RR^n)$ such that $w:= e^{it\De} w_0$
    approximates~$v$ on~$\Om'_T:=D'\times (-T,T)$ as
    \[
\|v-w\|_{L^2(\Om'_T)}\leq \ep M
    \]
    and $w_0$ satisfies the sharper bound
    \[
\|w_0\|_{L^2(\RR^n)}\leq e^{e^{C\ep^{-1/\si}}} M \,.
    \]
  \end{enumerate}
  Here the constant~$C$ depends on $T$ and on the geometry of the domains.
\end{theorem}

Let us recall that global approximation theorems are classical in the
case of elliptic and hypoelliptic operators~\cite{Browder,Hormander,
  Lax,Malgrange}, starting with the work of Runge in complex analysis.
These results have been recently extended to the related setting of
parabolic operators~\cite{Duke}. The case of dispersive equations,
however, is substantially different and presents new key technical
subtleties. The way we solve these difficulties in the above
non-quantitative approximation result hinges on a careful analysis of
integrals defined by Bessel functions with real and complex arguments.

A quantitative approximation theorem for elliptic equations was first
established by Salo and R\"uland in~\cite{Ruland}. Specifically, they
show that a function~$v$ satisfying a nice linear elliptic
equation (e.g., the Laplace equation) on a smooth bounded domain~$\Om$
can be approximated in~$L^2$ by solutions to the elliptic
equation on a larger domain $\Om_1$ whose $L^2$~norm is controlled in
terms of the $H^1(\Om)$~norm of~$v$ and the geometry of the domains. The
gist of the proof is a stability argument, which boils down to a
three-sphere estimate: if the $L^2$~norm of a solution to the elliptic
equation is of order~1 over the ball $B_2$ and small over the
ball~$B_{1/2}$, then the $L^2$~norm of the solution on the ball~$B_1$
is small too. Here $B_r$ denotes the ball centered at the origin of
radius~$r$. Although one has effective Carleman estimates for the
Schr\"odinger equation~\cite{KS,Isakov}, this kind of three-sphere inequalities do not hold for the Schr\"odinger
equation. Indeed, just as in the case of the heat equation~\cite{Escauriaza}, one can
construct counterexamples using that, for each $\al>1$, a Tychonov-type argument shows the
existence of smooth solutions to the Schr\"odinger equation
on~$\RR^n\times\RR$ that are bounded as
\[
|u(x,t)|< C e^{C |x|^2/t - t^{-\al}/C}
\]
if $t>0$ and vanish for $t\leq 0$ (see e.g.~\cite[Exercise 2.24]{Tao}). Hence the question of which spacetime domains
possess some kind of quantitative approximation property for the
Schr\"odinger equation remains open.

The proof of our quantitative approximation theorem exploits the
connection, via the time Fourier transform, between solutions to the
Schr\"odinger equation on $D\times\RR$ and solutions to the
Helmholtz--Yukawa equation  on~$D$ (that is, the equation
\[
\De \vp-\tau\vp =0
\]
where the constant $\tau$ can be positive or negative). The first part of
the proof, which is of considerable interest in itself
(Theorem~\ref{T.Helmholtz}), consists of a quantitative approximation
theorem for the Helmholtz--Yukawa equation with frequency-dependent
control of a suitable norm of the global solution on the whole
space~$\RR^n$. The norm we control in this elliptic approximation
theorem is a natural one: it essentially reduces to the
Agmon--H\"ormander seminorm when $\tau<0$ and is a suitable
generalization thereof for $\tau>0$. One should note that, unlike the
R\"uland--Salo quantitative approximation~\cite{Ruland}, where the approximating
solution is only assumed to be defined on a larger but still bounded
domain, here we strive to estimate solutions that are
defined on the whole space~$\RR^n$, and also to keep track of the
dependence of the constants on the ``frequency''~$\tau$.
The second part of the proof of Theorem~\ref{T.qGAT} is
again based on careful manipulations of Bessel functions, which are in
particular employed to replace (modulo small errors) an initial
datum that grows exponentially fast at infinity by a Schwartz one.

\subsection{Organization of the paper}


In Section~\ref{S.Helmholtz} we start by proving a quantitative global
approximation theorem for the Helmholtz--Yukawa equation with
constants that depend on the frequency. In
Section~\ref{S.lemma} we provide a lemma on
non-quantitative approximation for the Schr\"odinger equation that
permits to approximate a solution on a general bounded spacetime
set~$\Om$ as in Theorem~\ref{T.GAT} by a solution on a larger
spacetime cylinder. The proofs of Theorems~\ref{T.GAT}
and~\ref{T.qGAT} are presented in Section~\ref{S.global}. Global
approximation theorems
are crucially used in the proof of our result on vortex reconnection (Theorem~\ref{T.GP}), which is given
in Section~\ref{S.VR}. In Section~\ref{S.observed} we discuss how this
scenario of vortex reconnection presents the key features observed in
the physics literature. In particular, we note that when one considers
solutions to the Gross--Pitaevskii equation that fall off at infinity,
which physically corresponds to the more flexible case of laser beams,
we can prove a somewhat stronger result. The case of the
Gross--Pitaevskii equation on the torus is considered too. The paper
concludes with a short Appendix where we present some calculations
with Bessel functions.

\section{Frequency-dependent global approximation for the
  Helmholtz--Yukawa equation}
\label{S.Helmholtz}

Our objective in this section is to obtain a global approximation
result for the Helmholtz--Yukawa equation. In addition to its
intrinsic interest, this is a key ingredient in the proof of Theorem~\ref{T.qGAT}. There are two aspects that we need to pay attention
to. Firstly, we need to control the dependence on the
frequency~$\tau$. Secondly, we aim to control natural norms of the
global solution over the whole space~$\RR^n$. Since some functions
(such as the fundamental solutions we consider) take a slightly
different form when $n=2$, for the ease of
notation, we will hereafter assume that $n\geq3$. The case $n=2$
only involves minor modifications and can be tackled using the same arguments.

We will first prove several auxiliary lemmas. We start with a stability
lemma where the key point is the explicit dependence of the stability constants
on~$\tau$. Here and in what follows, we will use the notation
\[
  \tau_\pm:=\frac12(|\tau|\pm\tau)
\]
for the positive and negative parts of
the real number~$\tau$. Also, for the ease of notation, given an open set~$D$, we often denote the $L^2(D)$~norm of a
function~$f$ by $\|f\|_D$. Throughout, $B_R$ denotes the ball centered
at the origin of radius~$R$.

\begin{lemma}\label{L.3sphere}
Suppose that the function $\psi\in H^1(D)$ satisfies the elliptic
equation
\[
\De\vp -\tau \vp =0
\]
in a smooth bounded domain~$D$, where $\tau$ is a real constant. If~$D'$ is
another smooth domain whose
closure is contained in~$D$, then the following global stability
estimate holds:
\begin{equation}\label{stab}
\|\vp\|_{L^2(D)}\leq C e^{C\sq}\|\vp\|_{H^1(D)}\, \log^{-\mu}\frac{\|\vp\|_{H^1(D)}}{\|\vp\|_{L^2(D')}}\,.
\end{equation}
Likewise, if $D'\subset\!\subset D''\subset\!\subset D$ is a bounded
domain with a smooth boundary, we have the interior stability
inequality
\begin{equation}\label{stab2}
\|\vp\|_{L^2(D'')}\leq C e^{C\sq}\|\vp\|_{L^2(D)}^\te\|\vp\|_{L^2(D')}^{1-\te}\,.
\end{equation}
Here $C$, $\mu$ and~$\te$ are positive constants that do not depend on~$\tau$.
\end{lemma}

\begin{proof}
The key ingredient of these stability inequalities is to control the
dependence on~$\tau$ of the 3-sphere
estimate
\begin{equation}\label{3sph}
\|\vp\|_{B_{R_2}}\leq Ce^{C\sq} \|\vp\|_{B_{R_1}}^\te \|\vp\|_{B_{R_3}}^{1-\te}\,.
\end{equation}
We will shortly show that this inequality holds for concentric balls of radii $R_1<R_2<R_3$ contained
in~$D$ for some~$\te>0$ depending on the radii but not
on~$\tau$. One this 3-sphere inequality has been established, a standard argument of propagation of
smallness~\cite{Alessandrini} yields the interior stability
result~\eqref{stab2}, and also the stability
estimate up to the boundary~\eqref{stab}.

To derive the basic estimate~\eqref{stab}, we notice that when
$\tau<0$, the estimate~\eqref{3sph} was proved by
Donnelly--Fefferman~\cite{Donnelly}. If $\tau>0$, we use that
$\psi(x,t):= e^{\tau t} \vp(x)$ satisfies the heat equation
\[
\pd_t\psi-\De\psi=0
\]
on $D\times\RR$. A result of Escauriaza--Vessella~\cite{Escauriaza}
then shows
\[
\int_{B_{R_2}}|\psi(x,0)|^2\, dx\leq
C\bigg(\int_{B_{R_1}}|\psi(x,0)|^2\, dx\bigg)^\te
  \bigg(\int_{-c}^0\int_{B_{R_3}}|\psi(x,t)|^2\, dx\, dt\bigg)^{1-\te}
\]
for some positive constants $C,c$, which translates into
\[
\|\vp\|_{B_{R_2}}\leq C
\bigg(\frac{1-e^{-2c\tau}}{2\tau}\bigg)^{\frac{1-\te}2}\|\vp\|_{B_{R_1}}^\te \|\vp\|_{B_{R_3}}^{1-\te}\,.
\]
This readily implies
\[
\|\vp\|_{B_{R_2}}\leq C \|\vp\|_{B_{R_1}}^\te \|\vp\|_{B_{R_3}}^{1-\te}\,,
\]
thereby completing the proof of the 3-sphere inequality~\eqref{3sph}.
\end{proof}

In the following lemma we compute a fundamental solution for the
operator $\De-\tau$ with the sharp decay at infinity:

\begin{proposition}\label{P.fund}
Suppose that $n\geq3$. The function
\begin{equation}\label{defG}
  G_\tau(x):=\begin{cases}
    \be_n\, \tau^{\frac{n-2}4}\, \dfrac{K_{\frac n2-1}(\tau^{1/2} |x|)}{|x|^{\frac n2-1}} & \text{if }
    \tau> 0\,,\\[2mm]
    \be_n' |x|^{2-n}& \text{if }
    \tau= 0\,,\\[2mm]
    \be_n''\, |\tau|^{\frac{n-2}4}\, \dfrac{Y_{\frac n2-1}(\sqa |x|)}{|x|^{\frac n2-1}} & \text{if } \tau<0
  \end{cases}
\end{equation}
satisfies the distributional equation
\begin{equation}\label{fundsol}
\De G_\tau-\tau G_\tau=\de_0
\end{equation}
on~$\RR^n$.  Here $Y_\nu$ and~$K_\nu$ denote the Bessel function and
the modified Bessel function of the second kind, respectively, and the
normalization constants depend on the dimension. Furthermore, $G_\tau$
can be written as
\[
G_\tau(x)=|x|^{2-n} H_\tau(x)\,,
\]
where $H_\tau$
is bounded as
\begin{equation}\label{boundG}
| H_\tau(x)|\leq C_n 
\langle\sqa|x|\rangle^{\frac{n-3}2}\, e^{-\sqP|x|}\,.
\end{equation}
\end{proposition}

\begin{proof}
A straightforward computation in spherical coordinates shows that
$G_\tau(x)$ satisfies the equation
\[
  \De G_\tau-\tau G_\tau=0
\]
on~$\RR^n\backslash \{0\}$. In view of the asymptotic behavior of the Bessel functions at~0 that~$G_\tau$, the dimensional constants can be chosen so
that
\[
G_\tau(x)=\frac1{|\SS^{n-1}||x|^{n-2}}+ O(|x|^{3-n})\,,
\]
so it is standard that it satisfies the distributional
equation~\eqref{fundsol}.

To estimate~$G_\tau$, recall that the Bessel functions $Y_\nu$ and
$K_\nu$ are bounded for $r>0$ as
\[
|Y_\nu(r)|\leq C_\nu\frac{\langle r\rangle^{\nu-\frac12}}{r^\nu}\,,\qquad 
|K_\nu(r)|\leq C_\nu \frac{\langle r\rangle^{\nu-\frac12}}{r^\nu}\, e^{-r}\,.
\]
Plugging these estimates into the expression for~$G_\tau$ results
in~\eqref{boundG}.
\end{proof}

We shall also need frequency-dependent estimates for the convolution
of the fundamental solution with a compactly supported function:

\begin{lemma}\label{L.consts}
Let $w:= G_\tau* f$ with a function $f$ supported on a bounded
domain~$Y\subset\RR^n$. Given any bounded domain $B\subset\RR^n$, one
has
\[
\|w\|_{L^2(B)}+ \langle\tau\rangle^{-\frac12} \|w\|_{H^1(B)} +
\langle\tau\rangle^{-1} \|w\|_{H^2(B)}\leq C\langle\tau_-\rangle^{\frac{n-3}4}\,\|f\|_{L^2(Y)}\,,
\]
where the constant~$C$ depends on $n$, $B$ and~$Y$ but not on~~$\tau$.
\end{lemma}

\begin{proof}
  The bound~\eqref{boundG} implies that
  \[
\sup_{x\in B,\; y\in Y} |H_\tau(x-y)|\leq C\langle\tau_-\rangle^{\frac{n-3}4}\,,
  \]
so  we readily  obtain, for all $x\in B$,
  \[
|w(x)|\leq \int_Y \frac{|H_\tau(x-y)\, f(y)|}{|x-y|^{n-2}}\, dy\leq C\langle\tau_-\rangle^{\frac{n-3}4}\int_Y \frac{ |f(y)|}{|x-y|^{n-2}}\, dy\,.
  \]
  Standard estimates for Riesz potentials then yield
  \[
\|w\|_B\leq C\langle\tau_-\rangle^{\frac{n-3}4}\|f\|_Y\,.
  \]
To estimate the second derivatives of~$w$, we use the equation $\De w= \tau
w+ f$, obtaining
\[
\| w\|_{H^2(B)}=\|w\|_B+ \|\De w\|_B\leq C |\tau|\|w\|_B +C \|f\|_Y\leq C \langle\tau\rangle \langle\tau_-\rangle^{\frac{n-3}4}\|f\|_Y\,.
\]
The estimate for~$\| w\|_{H^1(B)}$ then follows by interpolation.
\end{proof}

The norm that we will employ to control the growth of solutions to the
Helmholtz--Yukawa equation at infinity (with $\tau\neq0$) is
\[
\triple \vp:= \limsup_{R\to\infty} \bigg(\frac1R\int_{B_R}|
\vp(x)|^2\,e^{-2\sqP|x|} \, dx\bigg)^{1/2}\,.
\]
To motivate the choice of this norm, recall that the natural norm to control solutions to the
Helmholtz equation
\[
\De\vp+\la^2 \vp=0
\]
on~$\RR^n$ is the Agmon--H\"ormander seminorm~\cite{Hormander,Vega2}
\begin{equation}\label{AH}
\limsup_{R\to\infty} \frac1R\int_{B_R}|
\vp(x)|^2\, dx\,.
\end{equation}
Indeed, a solution to the Helmholtz equation with sharp decay can be
written as the Fourier transform of a measure on the sphere with an
$L^2$ density,
\begin{equation}\label{vpx}
\vp(x)=\int_{\SS^{n-1}} e^{i\xi\cdot x}\,f(\xi)\, d\si(\xi)\,,
\end{equation}
and the norm $\|f\|_{L^2(\SS^{n-1})}$ turns out to be equivalent
to~\eqref{AH}~\cite[Theorem 7.1.28]{Hormander}. Since we are also concerned with solutions to the
Yukawa equation (that is, the Helmholtz equation with a negative
sign), one can replace the
definition~\eqref{AH} as above (which is also sharp when $\tau>0$).
It is standard that the only solution to the equation
\[
\De\vp-\tau \vp=0
\]
on~$\RR^n$ whose associated seminorm $\triple \vp$ is zero is the trivial solution
$\vp=0$. For technical reasons, we also define the weighted $L^\infty$~norm
\[
\triplei\vp:= \sup_{x\in\RR^n} \langle x\rangle^{\frac {n-1}2}
e^{-\sqP|x|} |\vp(x)|\,.
\]
Obviously this is a stronger norm, as
\[
\triple \vp \leq C\triplei \vp
\]
for any function~$\vp$.

The main result of this section, which builds upon ideas of
R\"uland--Salo~\cite{Ruland}, is the following global approximation
theorem for the Helmholtz--Yukawa equation. It should be stressed that
the reason for which the statement is more technically involved than
one would have liked is that we want to control both the case of large
$|\tau|$ (which works just fine) and the case of small~$|\tau|$. While
the latter case does not present any essential difficulties, it is
awkward to write estimates that are uniform in~$\tau$. This is due to
the fact that the triple norm obviously
collapses in the case $\tau=0$ (i.e., because harmonic functions do
not decay on average at infinity). This leads to the introduction of a
constant~$\tau_1$ to write estimates that we can conveniently invoke in later
sections. Still, the information that one can prove in the case of small~$|\tau|$ (and
in particular in the case of harmonic functions) is more precise than
what we state here, so the interested reader should take a look at
Step~4 of the proof.

\begin{theorem}\label{T.Helmholtz}
  Let $\vp\in H^1(D)$ satisfy the equation
  \[
    \De\vp-\tau\vp=0
  \]
  in a bounded domain~$D\subset\RR^n$. Let us fix some $\tau_1>0$. Assume that the complement $D\co$ is
  connected and set
  \begin{equation}\label{defNep}
N_{\ep,\tau}:= \exp\frac{C
  \langle\tau\rangle^{\frac12}e^{C\sq}}\ep\,,\qquad \tN_{\ep,\tau}:=
\bigg(\frac{\langle\tau\rangle}\ep\bigg)^C e^{C\sq}\,.
\end{equation}
Then, for each $\ep\in(0,1)$, one can find a solution of the equation
$\De\psi-\tau\psi=0$ on~$\RR^n$ such that:
  \begin{enumerate}
    \item If $|\tau|>\tau_1$, $\psi$ approximates~$\vp$ in the whole domain
  \[
\|\vp -\psi\|_{L^2(D)}\leq \ep \|\vp\|_{L^2(D)}^{1/2}\|\vp\|_{H^1(D)}^{1/2}
  \]
and is bounded as
  \[
\triplei{\psi}\leq (N_{\ep,\tau})^{ N_{\ep,\tau}}\|\vp\|_{L^2(D)}\,.
\]

\item Given a smaller subset $D'\ssubset D$,  if $|\tau|>\tau_1$, $\psi$ approximates~$\vp$ on~$D'$ as
  \[
\|\vp -\psi\|_{L^2(D')}\leq \ep \|\vp\|_{L^2(D')}^{1/2}\|\vp\|_{H^1(D')}^{1/2}
\]
and is bounded as
\[
\triplei{\psi}\leq (\tN_{\ep,\tau})^{\tN_{\ep,\tau}}\|\vp\|_{L^2(D)}\,.
\]

\item If $|\tau|\leq\tau_1$, $\psi$ approximates~$\vp$ on the whole domain as
  \[
\|\vp -\psi\|_{L^2(D)}\leq \ep \|\vp\|_{L^2(D)}^{1/2}\|\vp\|_{H^1(D)}^{1/2}
\]
and is bounded as
\[
|\psi(x)|\leq  (N_{\ep,1}\langle x\rangle)^{N_{\ep,1}}e^{\sqP |x|}\|\vp\|_{L^2(D)}\,.
\]
Furthermore, $\triplei\psi<\infty$ for
all~$\tau\neq0$.

\item  Given a smaller subset $D'\ssubset D$, if $|\tau|\leq\tau_1$, $\psi$ approximates~$\vp$ on~$D'$ as
  \[
\|\vp -\psi\|_{L^2(D')}\leq \ep \|\vp\|_{L^2(D')}^{1/2}\|\vp\|_{H^1(D')}^{1/2}
\]
and is bounded as
\[
|\psi(x)|\leq  (\tN_{\ep,1}\langle x\rangle)^{\tN_{\ep,1}}e^{\sqP |x|}\|\vp\|_{L^2(D)}\,.
\]
Furthermore, $\triplei\psi<\infty$ for
all~$\tau\neq0$.
\end{enumerate}
The constants only depend on the domains~$D$ and~$D'$ and on~$\tau_1$.
\end{theorem}

\begin{proof}
  Let us start by assuming that $|\tau|>\tau_1$ and proving the estimates up to the boundary, which
  correspond to the first item of the statement. The estimates in the
  smaller domain~$D'$ will be tackled in Step~3, while uniform
  estimates for
  $|\tau|\leq\tau_1$ will be presented in Step~4.\smallskip
  
\step{Step 1: Approximation by solutions with localized
    sources}  Let us consider a ball~$B$ containing the closure of~$D$
  and fix a bounded open set $ Y \subset\!\subset
  \RR^n\backslash \overline B$. Let us define the space of solutions
  \[
\cX:=\{ \phi \in L^2(D): \De\phi-\tau \phi=0\}\,.
  \]

 Consider the linear operator $L^2(Y)\to \cX$ defined by
  \begin{equation}\label{map}
f\mapsto G_\tau * f|_D
\end{equation}
and denote by~$\cK$ its kernel. We are interested in the
orthogonal complement
\[
\cKp :=\bigg\{ f\in L^2(Y): \int_Y f\overline g\, dx=0 \text{ for all }
g\in \cK\bigg\}\,,
\]
so we denote by $\cP: L^2(Y)\to \cKp$ the orthogonal projection. Let us call $A: \cKp\to \cX$ the restriction of the linear
map~\eqref{map} to this set. If $\phi$ is any function in~$L^2(D)$ and
$f\in L^2(Y)$, it
is clear that
\[
\int_D \phi(x)\, (G_\tau*f)(x)\, dx=\int_Y f(y)\, (G_\tau* \phi)(y)\,
dy\,,
\]
so the adjoint $A^*: \cX\to \cKp$ of~$A$ must be given by
\[
A^*\phi= \cP(G_\tau* \phi|_Y)\,.
\]
Furthermore, for $\phi\in L^2(D)$, if $g\in \cK$ one obviously has
\[
  \int_Y g(y)\, (G_\tau* \phi)(y)\, dy=\int_D \phi(x)\,
  (G_\tau*g)(x)\, dx =0\,,
\]
so one can drop the projector in the above expression for~$A^*$ and
simply write
\[
A^*\phi=G_\tau* \phi|_Y\,.
\]

The map $A^*A$ is a positive, compact, self-adjoint operator on
$\cKp$. Furthermore, the range of~$A$ is dense on~$\cX$ by standard non-quantitative
global approximation theorems for elliptic equations~\cite{Browder}. It then follows that there are positive
constants~$\al_j$ and orthonormal bases
$\{f_j\}_{j=1}^\infty$ of~$\cKp$ and $\{\phi_j\}_{j=1}^\infty$
of~$\cX$ such that
\[
Af_j=\al_j\phi_j\,,\qquad A^*\phi_j=\al_jf_j\,.
\]
Given a solution $\vp\in \cX$, let us consider its decomposition
\[
\vp= \sum_{j=1}^\infty \be_j \phi_j
\]
and, for any~$\al>0$, define the element of~$\cKp$
\[
F:= \sum_{\{j\; :\; \al_j > \al\}} \frac{\be_j}{\al_j} f_j\,.
\]
Note that obviously
\begin{equation}\label{normF}
\|F\|_Y = \bigg(\sum_{\{j\,: \,\al_j<\al\}} \frac{|\be_j|^2}{\al_j^2}
\bigg)^{1/2}\leq \frac{\|\vp\|_D}\al\,.
\end{equation}

As $\al\to0$, it is clear that $AF$ defines an approximation
of~$\vp$. To control how accurate this approximation is, let
\[
E:= \vp-AF = \sum_{\{j\; :\; \al_j < \al\}} \be_j \phi_j\in \cX
\]
denote the error, which is supported on~$D$. It is obviously bounded as
\begin{equation}\label{Erough}
  \|E\|_{D}= \bigg(\sum_{\{j\,: \,\al_j<\al\}} |\be_j|^2 \bigg)^{1/2}\leq \|\vp\|_{D}\,.
\end{equation}
In order to derive better
bounds, let us consider the function $w:=G_\tau* E$, which satisfies
the equation
\[
\De w-\tau w= E\,.
\]
Its $L^2$~norm on~$Y$ is obviously
given by
\begin{equation}\label{wY}
\|w\|_{Y}^2= \|A^*E\|_{Y}^2=\sum_{\{j\; :\; \al_j <
  \al\}} \al_j^2|\be_j|^2\leq \al^2\|E\|_{D}^2\,.
\end{equation}
We now estimate
\begin{align*}
\|E\|_D ^2 &=\int_D E\, (\De w-\tau w)\, dx= \int_D \vp\, (\De w-\tau
             w)\, dx = \int_{\pd D} (\vp\,\pd_\nu w- w\,\pd_\nu \vp)\,
             d\si\\
  &\leq C\|\vp\|_{H^{\frac12}(\pd D)} \|w\|_{H^{\frac12}(\pd D)}\leq C \|\vp\|_{H^1(D)}\|w\|_{H^{\frac12}(\pd D)}\,.
\end{align*}
In passing to the second equality we have used that $E=\De w-\tau w$
is orthogonal to $AF= \vp-E$. To pass to the third equality we
integrate by parts and use that $\De\vp-\tau\vp=0$. For the last
estimate we use the trace inequality.

To control the norm $\|w\|_{H^{\frac12}(\pd D)}$, we can use the trace
inequality and take a domain~$B'$ containing $\overline D\cup Y$ to
write
\begin{equation*}
\|w\|_{H^{\frac12}(\pd D)}\leq C\|w\|_{H^1(B'\backslash D)}\leq C\|w\|_{L^2(B'\backslash D)}^{\frac12}\|w\|_{H^2(B'\backslash D)}^{\frac12}\,.
\end{equation*}
For the last inequality we have simply interpolated the $H^1$~norm
of~$w$. As $w= G_\tau* E$, the $H^2$~norm of~$w$ can be estimated using
Lemma~\ref{L.consts} and the rough bound~\eqref{Erough} as
\begin{equation}\label{wH2}
\|w\|_{H^2(B')}\leq C
\langle\tau\rangle\langle\tau_-\rangle^{\frac{n-3}4}\|E\|_{D}\leq C \langle\tau\rangle\langle\tau_-\rangle^{\frac{n-3}4}\|\vp\|_{D}\,.
\end{equation}

For the $L^2(B'\backslash D)$~norm of~$w$ we use a finer argument
that employs the stability estimate stated in
Lemma~\ref{L.3sphere}. The starting point is that~$w$ satisfies the
equation $\De w-\tau w=0$ on~$B'\backslash\overline{D}$ and that its
$L^2$~norm on the subset $Y\subset B'\backslash D$ is small by~\eqref{wY}. Hence one can write
\begin{align}
\|w\|_{B'\backslash D}& \leq C e^{C\sq}\|w\|_{H^1(B'\backslash D)}\,
                        \log^{-\mu}\frac{\|w\|_{H^1(B'\backslash D
                        )}}{\|w\|_{Y}}\label{estimw}\\
  &\leq C
    e^{C\sq}\langle\tau\rangle^{\frac12}\langle\tau_-\rangle^{\frac{n-3}4}\|E\|_{D}\log^{-\mu}\frac{C
    \langle\tau\rangle^{\frac12}\langle\tau_-\rangle^{\frac{n-3}4}}{\al
    }\notag\\
  &\leq C e^{C\sq}\langle\tau\rangle^{\frac12}\|\vp\|_{D}\log^{-\mu}\frac{C \langle\tau\rangle^{\frac12}\langle\tau_-\rangle^{\frac{n-3}4}}{\al }\,. \notag
\end{align}
Here we have also used the bounds for the $H^1$~norm of~$w= G_\tau*E$ derived in
Lemma~\ref{L.consts}.

Putting all the estimates together, we have shown that for any~$\al>0$
one can take a function $F\in L^2(Y)$ such that $w:= G_\tau*F$ is
bounded as in~\eqref{normF} and approximates~$\vp$ in~$D$ as
\[
\|w-\vp\|_D\leq C e^{C\sq}\langle\tau\rangle^{\frac12}\|\vp\|_{D}^{\frac12}\|\vp\|_{H^1(D)}^{\frac12}\log^{-\mu'}\frac{C \langle\tau\rangle^{\frac12}\langle\tau_-\rangle^{\frac{n-3}4}}{\al }\,,
\]
with $\mu'>0$. Equivalently, for any $\ep>0$ one can choose~$F$ as
above such that $w:= G_\tau *F$ approximates~$\vp$ as
\begin{equation}\label{wvp}
\|w-\vp\|_D\leq \ep \|\vp\|_{D}^{\frac12}\|\vp\|_{H^1(D)}^{\frac12}
\end{equation}
and is bounded as
\begin{equation}\label{FY}
\|F\|_Y\leq C\|\vp\|_D\exp\frac{C \langle\tau\rangle^{\frac12}e^{C\sq}}\ep\,.
\end{equation}

\step{Step~2: Approximation by a global solution with controlled
  behavior at infinity} Let us start by choosing a slightly smaller
ball $B'' \subset\!\subset B$. Since $w:= G_\tau*F$
satisfies the equation
\[
\De w-\tau w=0
\]
on~$B$, it is then a straightforward consequence of
Lemma~\ref{L.consts} that
\begin{equation}\label{wlmHs}
\|w\|_{H^s(B'')}\leq C_s\langle\tau\rangle^{s/2} \|w\|_B\leq C_s\langle\tau\rangle^{s/2}\langle\tau_-\rangle^{\frac{n-3}4}\|F\|_Y
\end{equation}
for any nonnegative~$s$. For convenience, let us fix a ball~$B''$ as
above and denote its radius by~$R''$.

Consider an orthonormal basis of spherical harmonics on the unit
sphere~$\SS^{n-1}$ of dimension $(n-1)$, which we denote by
$Y_{lm}$. Hence $Y_{lm}$ is an eigenfunction of the spherical
Laplacian satisfying
\begin{equation}\label{eigenv}
\De_{\SS^{n-1}} Y_{lm}= -l(l+n-2)\, Y_{lm}\,,
\end{equation}
$l$ is a
nonnegative integer and $m$ ranges from~1 to the multiplicity
\[
d_l := \frac{2l+n-2}{l+n-2} \binom{l+n-2}{l}
\]
of the corresponding eigenvalue of the spherical Laplacian.

The expansion of~$w$ in spherical harmonics,
\[
w(x)= \sum_{l=0}^\infty \sum_{m=1}^{d_l} w_{lm}(|x|)\,
Y_{lm}\Big( \frac x{|x|} \Big) \,,
\]
converges in any $ H^s(B'')$ norm, with the coefficients being given by 
\[
w_{lm}(r):=\int_{\SS^{n-1}} w(r\om) \,
Y_{lm}(\om)\, d\si(\om)\,.
\]
Using the notation
\[
\|w_{lm}\|^2:= \int_0^{R''} |w_{lm}(r)|^2\, r^{n-1}\, dr\,,
\]
it is clear from~\eqref{wlmHs} and the expression for the spherical
eigenvalues (Equation~\eqref{eigenv}) that
\[
\sum_{l=0}^\infty\sum_{m=1}^{d_l} \L^{4s} \|w_{lm}\|^2\leq C \|w\|_{H^s(B'')}^2\leq C_s\langle\tau\rangle^{s}\langle\tau_-\rangle^{\frac{n-3}2}\|F\|_Y^2\,.
\]
Hence, for any~$s>1$ there is a positive constant~$C_s$
such that
\[
  \sum_{m=1}^{d_l} \|w_{lm}\|^2\leq C_s \L^{-4s}\langle\tau\rangle^{s}\langle\tau_-\rangle^{\frac{n-3}2}\|F\|_Y^2\,,
\]
which ensures that
\begin{align*}
\sum_{l=l_0}^\infty\sum_{m=1}^{d_l} \|w_{lm}\|^2 &\leq C_s
\langle\tau\rangle^{s} \langle\tau_-\rangle^{\frac{n-3}2}\|F\|_Y^2
                                                   \sum_{l=l_0}^\infty\L^{-4s}\\
  &\leq C_s
\langle l_0\rangle^{1-4s}\langle\tau\rangle^{s}
    \langle\tau_-\rangle^{\frac{n-3}2}\|F\|_Y^2\\
  &\leq C_s \|\vp\|_D^2
\langle l_0\rangle^{1-4s}\langle\tau\rangle^{s} \exp\frac{C \langle\tau\rangle^{\frac12}e^{C\sq}}\ep\,.
\end{align*}
Here we have employed the bound~\eqref{FY} for~$\|F\|_Y$. It then follows that, for any $\ep>0$, there are large
constants $c$ and~$s$, depending on~$\ep$ but not on~$\tau$,
such that if one sets
\begin{equation}\label{defl0}
l_0:= \exp\frac{c \langle\tau\rangle^{\frac12}e^{C\sq}}\ep
\end{equation}
and
\begin{equation}\label{defpsi}
\psi(x):=\sum_{l=0}^{l_0} \sum_{m=1}^{d_l} w_{lm}(|x|)\,
Y_{lm}\Big( \frac x{|x|} \Big) 
\end{equation}
one has
\begin{equation}\label{corte}
\|w-\psi\|_D^2=\sum_{l=l_0+1}^\infty\sum_{m=1}^{d_l} \|w_{lm}\|^2\leq \ep^2  \|\vp\|_D^2\,.
\end{equation}

Let us now pass to the study of the function~$w_{lm}$. As $w_{lm}$ satisfies the ODE
\begin{equation*}
\bigg(\pd_{rr} +\frac{n-1}r\pd_r -\frac{l(l+n-2)}{r^2} -\tau\bigg) w_{lm}(r)=0
\end{equation*}
and is bounded at $r=0$, one can write it in terms of a modified Bessel
function of the first kind as
\begin{equation}\label{hu}
w_{lm}(r)= A_{lm}\, r^{1-\frac n2} I_{l+\frac n2-1}(r\sqrt\tau)\,,
\end{equation}
where $A_{lm}$ is a $\tau$-dependent constant.

We need to obtain an estimate for the coefficient~$A_{lm}$. It is clear that
\[
A_{lm}= \frac{\int_0^{R''} r^{1-\frac n2}\, \overline{I_{l+\frac n2-1}(r\sqrt\tau)}\,
  w_{lm}(r)\, r^{n-1}\, dr}{\int_0^{R''} r\, |I_{l+\frac
    n2-1}(r\sqrt\tau)|^2\, dr}\,.
\]
If ${R''}$ denotes the radius of the ball~$B''$, the key point is then to
estimate the function
\begin{equation}\label{defcI}
\cI_\nu(\al):= \int_0^{R''} r\, |I_{\nu}(r\al )|^2\, dr\,,
\end{equation}
where we can assume that $\nu\geq1/2$ and $\al \in\RR^+\cup
i\RR^+$. Suitable upper and lower bounds have been computed in
Lemma~\ref{L.Bessel} in Appendix~\ref{A.Bessel}, showing in particular
that
\[
\cI_\nu(\al) \geq \frac C{\langle\al\rangle^2 \nu^2}\Big(\frac
{C\min\{|\al|,1\}}\nu\Big)^{2\nu} e^{C \Real\al}
\]
for all $\al\in \RR^+\cup i\RR^+$. This lower bound translates into the upper bound
\[
|A_{lm}|\leq \frac{\|w_{lm}\|}{\cI_{l+\frac
    n2-1}(\sqrt\tau)^{1/2}}\leq
{\|w_{lm}\|}{\langle\tau\rangle^{1/2} \langle l\rangle}e^{-C \sqP} \Big(\frac{l+\frac
    n2-1} {C\min\{\sqa,1\}}\Big)^{l+\frac
    n2-1}\,.
\]

The large time asymptotics of Bessel functions~\cite[8.451.1 and
8.451.5]{GR},
\begin{align*}
I_{l+\frac
  n2-1}(r\sqrt\tau)=\begin{cases}
\dfrac{\sqrt 2\cos(r\sqa - (2l+  n-1)\frac \pi4)+ O_l((\sqrt\tau r)^{-1})}{(\pi\sqa
  r)^{1/2}} & \text{if } \tau<0\,,\\
(2\pi)^{-1/2}\dfrac{e^{r\sqrt\tau}}{(\sqrt\tau
  r)^{1/2}} [1+ O_l((|\tau|^{1/2} r)^{-1})]& \text{if } \tau>0\,,
\end{cases}
\end{align*}
shows that
\[
\triple \psi\leq C\bigg(\sum_{l=0}^{l_0}\sum_{m=1}^{d_l} |A_{lm}|^2\bigg)^{1/2}\,.
\]
Note that the error terms $O_l((\sqp r)^{-1})$ are not bounded
uniformly in~$l$, but this does not have any effect on our argument.

To estimate $\psi$ pointwise, we can resort to the uniform estimate~\cite{Balodis}
\[
|J_\nu(z)|\leq C f_\nu(|\Real z|) \, e^{|\Imag z|}\,,
\]
where $C$ does not depend on~$\nu\geq1$ and
\[
  f_\nu(s):=\begin{cases}
    (\nu-s)^{-1} & \text{if } 0\leq s\leq \nu-\nu^{1/3}\,,\\
    \nu^{-1/3} & \text{if } \nu-\nu^{1/3}\leq s\leq \nu+\nu^{1/3}\,,\\
s^{-1/2}(1-\nu/s)^{-1/4} & \text{if } \nu+\nu^{1/3}\leq s\,.
  \end{cases}
\]
This bound, which obviously applies to the Bessel
function~$I_\nu(z)$ too because of the relation
\begin{equation}\label{JI}
J_\nu(iz)= i^\nu \, I_\nu(z)\,,
\end{equation}
readily shows
that~$\psi$ also satisfies the pointwise bound
\[
\triplei \psi \leq Cl_0^C\bigg(\sum_{l=0}^{l_0}\sum_{m=1}^{d_l} |A_{lm}|^2\bigg)^{1/2}\,.
\]
Here we have used the well-known bound
\[
\|Y_{lm}\|_{L^\infty(\SS^{n-1})}\leq C {d_l^{1/2}}\,.
\]
Putting together the bounds for $A_{lm}$, $\|w_{lm}\|$, $l_0$ and~$\|F\|_Y$, this
results in the bound provided in the first item of
the statement.\smallskip

\step{Step 3: The interior approximation result} We are left with the task of proving the quantitative approximation
result with sharper bounds on the smaller domain~$D'$, which
corresponds to the second assertion of the statement.

The proof goes as before with minor modifications. To derive an analog
of Step~1, all the
objects are defined with~$D'$ playing the role of~$D$; e.g., one sets
\[
\cX:=\{ \phi \in L^2(D'): \De\phi-\tau \phi=0\}\,.
\]
All the arguments remain valid without any further modifications until
one reaches the estimate~\eqref{estimw}. There one can now employ the
interior estimate in Lemma~\ref{L.3sphere}, which results in
\begin{align*}
\|w\|_{B'\backslash D}\leq C e^{C\sq} \|w\|_{B'\backslash D'}^{1-\te} \|w\|_Y^\te
\end{align*}
for some $\te\in(0,1)$. Since $\|w\|_Y\leq \al \|E\|_{D'}$, one can argue as before to obtain
\[
\|w\|_{B'\backslash D}\leq C \al^\te e^{C\sq}\|E\|_{D'}\,.
\]
This easily leads to the estimate
\begin{align*}
\|w-\vp\|_{D'}&\leq \ep \|\vp\|_{D'}^{\frac12}\|\vp\|_{H^1(D')}^{\frac12}\,,\\
\|F\|_Y&\leq C\bigg(\frac{\langle\tau\rangle}\ep\bigg)^{C}e^{C\sq}\|\vp\|_{D'}\,.
\end{align*}
Using the same reasoning as before with
\[
l_0:= C\bigg(\frac{\langle\tau\rangle}\ep\bigg)^{C}e^{C\sq}
\]
instead of~\eqref{defl0}, one arrives at the desired result.\smallskip

\step{Step~4: The case of small $|\tau|$} In the case of harmonic
functions ($\tau=0$), a minor variation of the proof shows that a
harmonic function on~$D$ can be approximated on this domain by a
harmonic polynomial of degree at most $ e^{C/\ep}$ whose
coefficients are bounded in absolute value by $e^{e^{C/\ep}}$. In
the approximation is performed on a smaller domain~$D'$, the
degree of the polynomial is bounded by $C\ep^{-C}$ and its coefficients are bounded by
$C\ep^{-\ep^{-C}}$. This translates into the bounds specified in the
statement. Likewise, in the case of small~$|\tau|$, one can reuse the
bounds for Bessel functions employed in the previous steps of
the proof to show that the bounds of the points~(iii) and~(iv) hold
uniformly in this case.
\end{proof}

\section{A lemma on non-quantitative approximation by solutions of the
  Schr\"odinger equation on a spacetime cylinder}
\label{S.lemma}

In this section we prove a lemma on approximation for the
Schr\"odinger equation that will be instrumental in our proof of Theorem~\ref{T.GAT}.

Before presenting this result, let us introduce some notation. If $L^2_c(\RR^{n+1})$ denotes the space of
$L^2$~functions on the spacetime with compact support, we can define
the operators
$\cT$ and $\cT^*$ as the maps $L^2_c(\RR^{n+1})\to L^\infty
L^2(\RR^{n+1})$ given by
\begin{align}\label{defT}
  \cT f(x,t)&:= \int_{-\infty}^t \int_{\RR^n} G(x-y,t-s)\, f(y,s)\, dy\, ds\,,\\
  \cT^*h(x,t)&:= \int_t^{\infty} \int_{\RR^n} G(x-y,t-s)\, h(y,s)\, dy\, ds\,,\label{defT*}
\end{align}
where
\[
  G(x,t):= (4\pi i t)^{-n/2} e^{i|x|^2/(4t)}
\]
is the fundamental solution of the Schr\"odinger equation. To put it differently,
\[
\cT f(\cdot, t)= \int_{-\infty}^t e^{i(t-s)\De} f(\cdot, s)\, ds\,,\qquad \cT^*h(\cdot, t)= \int_t^{\infty} e^{i(t-s)\De} h(\cdot, s)\, ds\,.
\]

Of course, the action of $\cT$ and $\cT^*$ is well defined
on more general distributions, which in this paper will always be of compact support. In the following proposition we recall some elementary properties of~$\cT$ and~$\cT^*$:

\begin{proposition}\label{P.TT*}
Given any $f,h\in L^2_c(\RR^{n+1})$, the following statements
hold:
\begin{enumerate}
\item For any  $t_0$ below the support of~$f$, $\cT f$ is the solution
  to the Cauchy problem
\begin{equation*}
(i\pd_t+\De) \cT f=if\,,\qquad \cT f|_{t=t_0}=0\,.
\end{equation*}

\item For any $t_0$ above the support of~$h$, $\cT^*h$ is the solution
  to the Cauchy problem
\begin{equation*}
(i\pd_t+\De) \cT^*h=-ih\,,\qquad \cT^*h|_{t=t_0}=0\,.
\end{equation*}

\item $\displaystyle \int_{\RR^{n+1}} \cT f(x,t)\, \overline{h(x,t)}\,
  dx\, dt=\int_{\RR^{n+1}} f(x,t)\, \overline{\cT^*h(x,t)}\,
  dx\, dt$.
  
\item $\cT$ and $\cT^*$ commute with spacetime derivatives. That is, if
$f,g\in C^\infty_c(\RR^{n+1})$,
\[
D^\al(\cT f)= \cT (D^\al f)\quad \text{and}\quad D^\al(\cT^*g)= \cT^*(D^\al g)
\]
for any multiindex~$\al$.
\end{enumerate}
\end{proposition}

\begin{proof}
 The first three statements follow by a straightforward computation. In order to see
that $D^\al (\cT f)= \cT (D^\al f)$, the easiest way is to differentiate the
equation satisfied by $\cT f$. This way we arrive at
\[
(i\pd_t+\De)D^\al (\cT f)= iD^\al f\,,\qquad D^\al(\cT f) |_{t=t_0}=0
\]
for any $t_0$ below the support of~$f$, which means that indeed $D^\al \cT f= \cT (D^\al
f)$. The proof that $D^\al(\cT^*g)= \cT^*(D^\al g)$ is analogous.
\end{proof}

To state the approximation result, it is also convenient to define
the notion of {\em admissible set}\/:

\begin{definition}\label{D.admissible}
  Let $\Om\subset\RR^{n+1}$ be an open set in spacetime whose closure is contained in the spacetime cylinder $\cC:=B_R\times
\RR$. We say that the set~$S\subset\RR^{n+1}$ is $(\Om,R)$-{\em
  admissible} if:
\begin{enumerate}
\item $S$ is a spacetime cylinder of the form $B\times (T_1,T_2)$ contained in~$\RR^{n+1}\backslash\overline\cC$, where $B\subset\RR^n$ is a ball.

\item The projection of~$S$ on the time axis contains that of~$\Om$,
  that is,
\begin{equation*}
\{t: (x,t)\in \Om \text{ for some } x\in\RR^n\}\ssubset \{t': (x',t')\in S \text{ for some } x'\in\RR^n\}\,.
\end{equation*}
\end{enumerate}
\end{definition}

We are now ready to state and prove the main result of this section:

\begin{lemma}\label{L.nonquant}
  Let $\Om\ssubset B_R\times \RR$ be a bounded domain in spacetime
  with smooth boundary such that the complement $\RR^n\backslash \Om_t$
  is connected for all~$t$. Fix any $(\Om,R)$-admissible set
  $S\subset\RR^{n+1}$ and a smaller domain $\Om'\ssubset\Om$. For some real~$s$, consider a function
  $v\in L^2H^s(\Om)$ that satisfies the Schr\"odinger equation
\[
i\pd_t v+ \De v=0
\]
in the distributional sense on~$\Om$. Then, for any $\ep>0$, there is
a function $f\in C^\infty_c(S)$ such that $\cT f$
approximates~$v$ in~$\Om'$ as
\[
\|v-\cT f\|_{L^2H^s(\Om')}<\ep\,.
\]
\end{lemma}

\begin{remark}
A cursory look at the proof reveals that the proposition remains valid for
solutions~$v$ in many other functional spaces, e.g.~in $H^s(\Om)$ or $C^k(\Om)$.
\end{remark}

\begin{proof}
 Since the bounded
  spacetime set $\Om \subset \RR^{n+1}$ has a smooth boundary, we can
  extend~$v|_\Om$ to a compactly supported function on~$\RR^{n+1}$ of
  class $L^2H^s(\RR^{n+1})$ by multiplying by a smooth cutoff
  function. We assume that this cutoff function is~1 in a small
  neighborhood~$U$ of the closure of~$\Om'$ and vanishes on the
  complement of~$\Om$. We will still denote this extension by~$v$.

Let us consider an approximation of the identity
$F_\eta(x,t):= \eta^{-n-1} F(x/\eta,t/\eta)$ defined by a function
$F\in C^\infty_c(\RR^{n+1})$ such that its support is contained in the
unit ball~$B_1$ and $\int_{\RR^{n+1}} F(x,t)\, dx\, dt=1$. Setting
$v^\eta:= F_\eta * v$, since $i\pd_t v+ \De v=0$ outside $\Om\backslash
U$ and $\supp F_\eta\subset B_\eta$, it follows that
\begin{align*}
(i\pd_t +\De ) v^\eta(x,t)&= \int_{\RR^{n+1}} F_\eta(y,s)\, (i\pd_t+\De)
                           v(x-y,t-s)\, dy\, ds=0
\end{align*}
whenever the distance from the point $(x,t)$ to the set $\Om\backslash
U$ is greater than~$\eta$. 

In particular, for small enough~$\eta$, $v^\eta$ is a smooth function,
whose support is contained in a neighborhood of~$\Om$ of width~$\eta$,
which satisfies the equation
\[i\pd_t v^\eta+\De v^\eta=0
\]
in an open neighborhood $\Om''\ssupset \Om' $ and whose
difference
\begin{equation}\label{vvep}
\|v-v^\eta\|_{L^2H^s(\RR^{n+1})}
\end{equation}
is as small as one wishes.

As $v^\eta(x,t_0)=0$ for all~$x$ and any negative enough
negative time~$t_0$, by defining the smooth, compactly supported function
\[
 f:=\pd_tv^\eta -i\De v^\eta \,,
\]
Proposition~\ref{P.TT*} ensures that one can write
\begin{equation}\label{v1}
v^\eta =\cT f\,,.
\end{equation}
Note that, by definition,
\[
  S':=\supp f
\]
is contained in a set of the form $\Om^\eta\backslash
\overline{\Om''}$, where~$\Om^\eta$ is the support of~$v^\eta$ (and is
therefore contained in an $\eta$-neighborhood of~$\Om$). Note that
\[
\dist (S',\Om' ):=\inf\{|X-Y|:X\in S',\; Y\in \Om' \}>0\,.
\]

We claim that, given any~$\ep>0$ and any positive integer~$k$, there
is a function $g$ in $C^\infty_c(S)$ such that
the function $\cT g$
approximates~$v^\eta$ on the set~$\Om' $ in the $L^2H^k$~norm
(cf.~Equation~\eqref{L2Hs}) as
\begin{equation}\label{v1v2}
\|v^\eta-\cT g\|_{L^2H^k(\Om' )}<\ep.
\end{equation}

To prove this, consider the space of smooth functions
\[
\cV:=\{ \cT g|_{\Om'} : g\in C^\infty_c(S)\}\,,
\]
which we consider as a subset of the Banach space $L^2H^k({\Om'}
)$. Since
\[
(i\pd_t +\De) \cT g= ig
\]
in~$\RR^{n+1}$ by Proposition~\ref{P.TT*}, in particular $(i\pd_t +\De) \cT g=0$ in the cylinder~$B_R\times\RR$. Now 
let $\phi$ be a distribution in the dual
space
\[
[L^2H^k({\Om'} )]'= L^2H^{-k}_0({\Om'} )\,,
\]
where the subscript refers to the vanishing trace conditions, such that
\begin{equation}\label{angle1}
0=\langle \phi, u\rangle_{\Om'} 
\end{equation}
for all $u\in \cV$. Here the angle bracket
$\langle\cdot,\cdot\rangle_{\Om'} $ denotes the duality coupling (which
we assume to be conjugate linear in the first entry). Notice
that, as $C^\infty({\Om'} )\subset L^2H^k({\Om'} )$, any element~$\phi$ of the
dual space $L^2H^{-k}_0({\Om'} )$ is a distribution on~$\RR^{n+1}$ whose
support is contained in~${\Om'} $. The duality bracket is then given simply
by the usual coupling $\langle\cdot,\cdot \rangle$ between a
distribution of compact support and a smooth enough function.

By the definition of~$\phi$ and Proposition~\ref{P.TT*}, it then follows that, for all $g\in C^\infty_c(S)$,
\begin{align*}
0= \langle \phi, \cT g\rangle= \langle \cT^*\phi, g\rangle\,,
\end{align*}
which implies that $\cT^*\phi\equiv 0$ in the interior of~$S$, where
(with some abuse of notation)
$\cT^*\phi$ denotes the distribution
\[
\cT^*\phi(\cdot,t):=\int_t^\infty e^{i(t-s)\De} \phi(\cdot,s)\, ds\,.
\]
Then $\cT^*\phi$ satisfies the Schr\"odinger equation
\[
(i\pd_t +\De) \cT^*\phi= -i\phi\,,
\]
and in particular
\[
(i\pd_t +\De) \cT^*\phi= 0
\]
in the complement of the set~${\Om'} $.

Let us now recall the unique
continuation property of the Schr\"odinger equation (see e.g.~\cite{Tataru}):
\begin{theorem}[Unique continuation property] \label{T.UCP}
  Let~$W\subset\RR^{n+1}$ be a connected open set in spacetime.  If a function satisfies the equation $i\pd_t h+\De h=0$ in~$W$ and $h=0$ in some nonempty open set~$V$ contained in~$W$,
  then $h(x,t)=0$ for all $(x,t)\in W$ such that the intersection
  $V_t\cap W$ is nonempty.
\end{theorem}
It then follows from the fact that
$\cT^*\phi\equiv 0$ on~$S$ and the connectedness of $\Om_t\co$ for all~$t$ that $\cT^*\phi(x,t)=0$ for all~$t$ such that
\[
S\cap (\RR^n\times\{t\})\neq \emptyset
\]
and all $x \in \RR^n$ such that $(x,t)\not\in {\Om'} $. By the time
projection condition that appears in Definition~\ref{D.admissible} and
the fact that $S'\subset \RR^{n+1}\backslash \Om'$, this implies that $\cT^*\phi=0$ on~$S'$, which
in turn means that
\[
0= \langle \cT^*\phi, f\rangle= \langle \phi, v^\eta\rangle\,.
\]
As a consequence of the Hahn--Banach theorem, $v^\eta$ can then be
approximated by an element $u\in\cV$ as in~\eqref{v1v2}, as we
wanted to conclude.

There is a slight technical point about the application of the unique
continuation theorem to $\cT^*\phi$ that we have
omitted. In~\cite{Tataru} this result is proved, for more general Schr\"odinger-type
equations, under the assumption that $\cT^*\phi\in L^2\loc$, which
does not hold in our case because $\phi$ is only in
$L^2H^{-k}_0$. However, this regularity assumption is not essential in
the case of the usual Schr\"odinger equation, as one can apply the
argument to the mollified function
\[
F_{\eta'}* (\cT^*\phi)\,,
\] 
with a small enough parameter~$\eta'$ and $F_{\eta'}$ as in
the beginning of the proof. The details are just as above. Lemma~\ref{L.nonquant}
then follows.
\end{proof}

\section{Global approximation theorems for the Schr\"odinger
  equation}
\label{S.global}

The quantitative and non-quantitative global approximation theorems for
the Schr\"odinger equation that we stated in the Introduction are
proved in this section. We start with the quantitative result, which
will readily give the non-quantitative one when combined with Lemma~\ref{L.nonquant}:

\begin{proof}[Proof of Theorem~\ref{T.qGAT}]

Let us begin with the case of approximation in the whole cylinder
$\Om=D\times\RR$, which corresponds to the first part of the
statement. For the sake of clarity, we will divide the proof in
several steps.\smallskip

\step{Step 1: Approximation by a global solution of the
  Schr\"odinger equation that grows at spatial infinity}
As~$v\in L^2(\Om)$ satisfies the Schr\"odinger equation in~$\Om$, its
time Fourier transform $\hv(x, \tau)$ must satisfy the
Helmholtz--Yukawa equation
\begin{equation}\label{HY}
\De \hv(x,\tau)-\tau \hv(x,\tau)=0
\end{equation}
on~$D$. Theorem~\ref{T.Helmholtz} ensures that for every $\ep'>0$ there is a
function~$\hpsi(x,\tau)$ that satisfies the equation
\[
\De\hpsi(x,\tau)-\tau\hpsi(x,\tau)=0
\]
on~$\RR^n$ and approximates $\hv(x,\tau)$ as
\begin{equation}\label{approx}
\|\hpsi(\cdot,\tau) - \hv(\cdot,\tau)\|_{L^2(D)}\leq \ep'
\|\hv(\cdot,\tau)\|_{L^2(D)}^{1/2}\|\hv(\cdot,\tau)\|_{H^1(D)}^{1/2}\leq C\ep'\langle\tau\rangle^{1/4}\|\hv(\cdot,\tau)\|_{L^2(D)}\,.
\end{equation}
Furthermore, $\hpsi(x,\tau)$ is bounded as
\begin{equation}\label{Bound1}
\triplei{\hpsi(\cdot,\tau)}\leq
(N_{\ep',\tau})^{N_{\ep',\tau}}\|\hv(\cdot,\tau)\|_{L^2(D)}
\end{equation}
if $|\tau|>\tau_1$ (where $\tau_1$ is some fixed constant) and as
\begin{equation}\label{Bound2}
|\hpsi(x,\tau)|^2\leq  (N_{\ep',1} \langle x\rangle)^{N_{\ep',1}}e^{\sqP|x|}\|\hv(\cdot,\tau)\|_{L^2(D)}
\end{equation}
if $|\tau|\leq \tau_1$. The quantity $N_{\ep,\tau}$ was
defined in~\eqref{defNep} and we are taking
\begin{equation}\label{defK}
\ep':= \ep^K
\end{equation}
for some constant~$K>1+1/(2\si)$. Let us also recall from the proof of
Theorem~\ref{T.Helmholtz} that~$\hpsi(x,\tau)$ is of the form
\begin{equation}\label{formulahpsi}
\hpsi(x,\tau)= \sum_{l=0}^{l_\tau} \sum_{m=1}^{d_l} 
A_{lm}(\tau)\, r^{1-\frac n2} I_{l+\frac n2-1}(r\sqrt\tau)\, Y_{lm}\Big( \frac x{|x|} \Big) \,,
\end{equation}
where $Y_{lm}$ are normalized spherical harmonics,
and that one has effective bounds (depending on~$\tau$ and~$\ep$) for the constants $A_{lm}(\tau)$ and~$l_\tau$.

Equation~\eqref{HY}, the approximation estimate~\eqref{approx} and the
hypothesis 
\[
\int_{|\tau|>\tau_0}\int_D |\hv(x,\tau)|^2\, dx\, d\tau\leq
M^2\langle\tau_0\rangle^{-\si}
\]
imply that one can take 
\[
\tau_\ep:= c\ep^{-2/\si}\,,
\]
with $c$ a large constant independent of~$\ep$, so that the function
\[
v_1(x,t):= \int_{|\tau|<\tau_\ep} e^{i\tau t} \hpsi(x,\tau)\, d\tau
\]
satisfies the Schr\"odinger equation
\[
i\pd_t v_1+\De v_1=0
\]
on~$\RR^{n+1}$ and approximates~$v$ as
\begin{align*}
\|v-v_1\|_{L^2(\Om)}^2 &=2\pi \int_{|\tau|>\tau_\ep}\int_D|\hv(x,\tau)|^2\, dx\,
d\tau + 2\pi\int_{|\tau|<\tau_\ep}\int_D |\hv(x,\tau)-\hpsi(x,\tau)|^2\, dx\,
                         d\tau \\
  &\leq M^2\tau_\ep^{-\si} + C\ep'
    \int_{|\tau|<\tau_\ep}\int_D\langle\tau\rangle^{1/2}|\hv(x,\tau)|^2\,
    dx\, d\tau\\
  &\leq C\ep^2 M^2\,.
\end{align*}
It follows from the formula~\eqref{formulahpsi} and the bounds derived
in Theorem~\ref{T.Helmholtz} that $v_1$ is a smooth
function satisfying the pointwise bound
\begin{equation}\label{brutal}
|v_1(x,t)|\leq (N_\ep)^{N_\ep} \exp\big( C\ep^{-\frac1{\si}}|x|\big)\,,
\end{equation}
with
\[
N_\ep:= \exp\exp \big(C\ep^{-\frac2{\si}}\big)\,.
\]


\smallskip

\step{Step~2: Approximation via compactly supported
  initial data} Let us now set
\[
  u_\de(x):= v_1(x,0)\, e^{-\de|x|^2}\,,
\]
where $\de>0$ is a small positive constant to be determined later, and set
\[
w(x,t):= e^{it\De} u_\de(x) \,.
\]
Note that $u_\de$ is
obviously in the Schwartz space~$\cS(\RR^n)$ by the
bounds~\eqref{Bound1}-\eqref{Bound2}. Therefore, using the
formula~\eqref{formulahpsi} and the integral formulation
\[
e^{it\De} u_\de(x)= \int_{\RR^n} G(x-y,t)\, u_\de(y)\, dy\,,
\]
one can write
\begin{equation}\label{uep}
  w(x,t)= (4\pi i t)^{-\frac
    n2}e^{i|x|^2/4t}\sum_{l=0}^{l_0} \sum_{m=1}^{d_l} \int_{|\tau|<\tau_\ep} A_{lm}(\tau)\,
  \cB_{lm}(x,t,\tau;\de)\,d\tau
  \,,
\end{equation}
where
\[
\cB_{lm}(x,t,\tau;\de):=\int_{\RR^n} e^{i\frac{|y|^2-2x\cdot y}{4t}-\de|y|^2}|y|^{1-\frac n2}\,Y_{lm}\Big(
  \frac y{|y|} \Big) \, I_{l+\frac n2-1}(|y|\sqrt\tau) \, dy\,.
\]

Let us integrate over the angular variables first. Denoting the unit
sphere by~$\SS^{n-1}$, let us now record the expression for the
Fourier transform of a spherical harmonic~\cite{Canzani}:
\[
\int_{\SS^{n-1}}Y_{lm}(\om)\, e^{-i\xi\cdot \om}\, d\si(\om) =  (-i)^l (2\pi)^{\frac
  n2} \, Y_{lm}\Big(
  \frac \xi{|\xi|} \Big) \, \frac{ J_{l+\frac
      n2-1}(|\xi|)}{|\xi|^{\frac n2-1}}\,.
\]
Using this formula and introducing spherical coordinates
\[
\rho:=|y|\in (0,\infty)\,,\qquad \om:= y/|y|\in \SS^{n-1}\,,
\]
the function
$\cB_{lm}(x,t,\tau;\de)$ can be readily written as
\begin{align*}
  \cB_{lm}(x,t,\tau;\de)&= \int_0^\infty e^{-(\de- \frac
                      i{4t})\rho^2}\rho^{\frac{n}2}\, I_{l+\frac n2-1}(\rho\sqrt\tau)\int_{\SS^{n-1}} Y_{lm}(\om)
                      \,e^{-i \frac{\rho x\cdot \om}{2t}} \,d\si(\om)\, d\rho\\
                    &= (- i)^l (2\pi)^{\frac
                      n2} \, Y_{lm}\Big(
                      \frac x{|x|} \Big) \, \cM_{lm}(|x|,t,\tau;\de)\,,
\end{align*}
where we have set, for $r>0$,
\[
  \cM_{lm}(r,t,\tau;\de):=\bigg(\frac{2t}{r}\bigg)^{\frac n2-1} \int_0^\infty e^{-(\de- \frac
                      i{4t})\rho^2}\rho\, I_{l+\frac
                      n2-1}(\rho\sqrt\tau)\, J_{l+\frac
                      n2-1}\Big(\frac{\rho r}{2t}\Big)\, d\rho\,.
                  \]
                  We can now use the formula~\cite[6.633.4]{GR} to compute the integral
in closed form:
\begin{multline*}
  \cM_{lm}(r,t,\tau;\de)=\\
  \frac{i(2t)^{\frac n2}}{1 + 4it\de} \,
\exp\bigg(\frac{-\de (r^2-4\tau t^2) + i(\tau t -\frac{r^2}{4t}) }{1+16\de^2t^2}\bigg)\, r^{1-\frac n2}\,
J_{l+\frac n2-1}\bigg(\frac{ir\sqrt\tau}{1+4it\de }\bigg)\,.
\end{multline*}

Plugging these formulas in the expression~\eqref{uep}  and using the identity between Bessel functions~\eqref{JI}, we readily find
that
\[
\lim_{\de\searrow0} w(x,t)= |x|^{1-\frac n2} \sum_{l=0}^{l_0} \sum_{m=1}^{d_l} Y_{lm}\Big(
                      \frac x{|x|} \Big) \int_{|\tau|<\tau_\ep}
                      e^{it\tau}\, A_{lm}(\tau)\,
  I_{l+\frac n2-1} (|x|\sqrt \tau)\,d\tau\,.
\]
In view of the formula for $v_1(x,t)$,
the above limit can be rewritten as
\begin{equation}\label{conv1}
\lim_{\de\searrow0} w(x,t)= v_1(x,t)
\end{equation}
uniformly for $(x,t)$ in any compact spacetime subset. 

To estimate the difference $v_1-e^{it\De} u_\de$ on $\Om_T=D\times
(-T,T)$, it suffices to notice that the dependence on the
parameter~$\de$ (which only appears in the function
$\cM_{lm}(r,t,\tau;\de)$) can be controlled using that
\begin{equation}\label{conv2}
|\cM_{lm}(r,t,\tau;\de_0)-\cM_{lm}(r,t,\tau;0)|\leq \de_0
\sup_{0\leq\de\leq\de_0}\bigg|\frac\pd{\pd\de}\cM_{lm}(r,t,\tau;\de)\bigg|\,.
\end{equation}
Therefore, using the bounds for Bessel functions as in Theorem~\ref{T.Helmholtz}
and the bounds for the constants $\tau_\ep$ and $A_{lm}(\tau)$, after
some straightforward manipulations one
obtains that there is some~$\de_T>0$, depending on~$T$, such
that
\begin{equation}\label{conv4}
\sup_{0\leq\de\leq\de_T}\sup_{0<r<R,\; |t|<T,\;
  |\tau|<\tau_\ep}\bigg|\frac\pd{\pd\de}w(x,t)\bigg|\leq
C_TM (N_\ep)^{N_\ep}\,.
\end{equation}
Here we have set
\[
N_\ep:= e^{e^{C\ep^{-1/\si}}}\,.
\]
One can therefore take some positive $\de_\ep$ of the form
\[
\de_\ep:= (N_\ep)^{-N_\ep}
\]
such that
\[
\|w-v_1\|_{L^2(\Om_T)}\leq C\ep M\,.
\]
The $L^2$~norm of the initial datum $u_\de$ can now be computed
using the pointwise bound~\eqref{brutal}:
\begin{align*}
\|u_\de\|_{L^2(\RR^n)}^2 &=\int_{\RR^n} e^{-2\de |x|^2}|v_1(x,0)|^2\,
                           dx\\
  &\leq (N_\ep)^{N_\ep} \int_{\RR^n} e^{-2\de
    |x|^2}\exp(C\ep^{-\frac1{\si}|x|})\, dx\\
  &\leq (N_\ep)^{N_\ep}\,.
\end{align*}
The first assertion of Theorem~\ref{T.qGAT}
then follows.\smallskip

\step{Step~3: The interior approximation estimate} The proof is
exactly the same but one uses interior estimates for the
function~$\hpsi(x,\tau)$, which result in the sharper bound for~$v_1$
\[
|v_1(x,t)|< N_\ep\exp(C\ep^{-\frac1{\si}}\langle x\rangle)\,.
\]
Substituting this bound in the integral for the $L^2$~norm of the
initial datum~$u_\de$, this yields
\[
\|u_\de\|_{L^2(\RR^n)} \leq C N_\ep\,,
\]
which completes the proof of the theorem.
\end{proof}

The following result is a straightforward variation of
Theorem~\ref{T.qGAT} where we impose additional regularity on the
local solution to control more derivatives of the functions involved. This is needed in the proof of Theorem~\ref{T.GAT}. For
concreteness, we only consider the interior case, which suffices for our purposes:

\begin{lemma}\label{L.higherreg}
  Let $\Om:= D\times \RR$, where $D\subset\RR^n$ is a bounded set with
  smooth boundary whose complement $\RR^n\backslash D$ is connected. Take $k\geq1$
  and fix some smaller set $D'\ssubset D$. Assume that
  $v\in L^2H^k(\Om)$ satisfies the Schr\"odinger equation~\eqref{Sch}
  in~$\Om$ and its time Fourier transform is bounded as
  \begin{equation}\label{conv3}
\int_{|\tau|>\tau_0}\int_D \langle\tau\rangle^k|\hv(x,\tau)|^2\, dx\, d\tau\leq
M^2\langle\tau_0\rangle^{-\si}
\end{equation}
for some $\si>0$ and all $\tau_0\geq0$. Then, for each $\ep\in
  (0,1)$ and any $T>0$, one can take an
      initial datum $u_0\in \cS(\RR^n)$ such that $u:= e^{it\De} u_0$
    approximates~$v$ on~$\Om'_T:=D'\times (-T,T)$ as
    \[
\|v-u\|_{L^2H^k(\Om'_T)}\leq \ep M
    \]
    and $u_0$ is bounded as
    \[
\|u_0\|_{H^k(\RR^n)}\leq  e^{e^{C \ep^{-\frac1{\si}}}}M\,.
    \]
 The constant~$C$
 depends on $k$,  $T$ and on the geometry of the domains.
\end{lemma}

\begin{proof}
The proof is just as in Theorem~\ref{T.qGAT} modulo minor
changes. Indeed, with the faster convergence rate that we have required on the
integral~\eqref{conv3} and the obvious estimate
\[
\|\hpsi(\cdot,\tau) - \hv(\cdot,\tau)\|_{H^k(D')}\leq
C\langle\tau\rangle^{k/2} \|\hpsi(\cdot,\tau) - \hv(\cdot,\tau)\|_{L^2(D)}\,,
\]
the function $v_1$, defined as above, is readily shown to
approximate~$v$ as
\[
\|v-v_1\|_{L^2H^k(\Om)}<CM\ep\,.
\]

One can define $u_\de$ as in the proof of Theorem~\eqref{T.qGAT} so that the approximation holds
in~$L^2H^k$. Indeed, it is not hard
to see that a
statement just like~\eqref{conv1} also holds when one takes spatial derivatives on both sides of
the equation. To obtain bounds, it suffices to replace the
estimate~\eqref{conv4} by
\[
\sum_{j+m\leq k}\sup_{0\leq\de\leq\de_T}\sup_{0<r<R,\; |t|<T,\;
  |\tau|<\tau_\ep}\langle l\rangle^j \bigg|\frac\pd{\pd\de}\nabla^mw(x,t)\bigg|\leq
C_TN_\ep M \,.
\]
The claim readily follows.
\end{proof}

The non-quantitative approximation theorem stated in the Introduction
is now an easy consequence of the results that we have already established:

\begin{proof}[Proof of Theorem~\ref{T.GAT}]
Assume that the set~$\Om$ is contained in 
$B_{R/2}\times (-R,R)$ and take an $(\Om,R)$-admissible set~$S$ (see Definition~\ref{D.admissible}). By
Lemma~\ref{L.nonquant}, there is a function~$f\in C^\infty_c(S)$ such
that the function $u:=\cT f$ approximates~$v$ as
\[
\|v-u\|_{L^2H^s(\Om')}<\frac\ep 2\,.
\]
Note that, by definition, this function satisfies the Schr\"odinger equation
\[
i\pd_t u+\De u=0
\]
in~$B_R\times\RR$. Furthermore, it follows from the fact that $f$~is compactly
supported and the expression of the fundamental solution~$G(x,t)$
that~$u$ is bounded as
\[
\sup_{x\in B_R}|\pd_t^N u(x,t)|< \frac C{\langle t\rangle^{\frac n2+N}}\,.
\]
Denoting by $\hu(x,\tau)$ the Fourier transform of~$u(x,t)$ with
respect to time, it then follows from the mapping properties of the
Fourier transform that for all $n\geq2$ and all $N\geq1$ one has
\begin{equation}\label{boundug}
\sup_{x\in B_R} \|\hv(x,\cdot)\|_{L^2(\RR)} + \sup_{x\in B,\;\tau\in\RR}|\tau^N
\hv(x,\tau)|< C\,,
\end{equation}
where the constant depends on~$N$. In view of this decay property of
the time Fourier transform of~$u$, Lemma~\ref{L.higherreg} ensures
that there exists an initial datum~$w_0\in \cS(\RR^n)$ such that $w:=
e^{it\De}w_0$ approximates~$u$ as
\[
\|u-w\|_{L^2H^k(B_{R/2}\times(-R,R))}<\frac\ep2\,.
\]
The theorem is then proved.
\end{proof}

\section{Vortex reconnection for the Gross--Pitaevskii equation}
\label{S.VR}

In this section we provide the proof of the result on vortex
reconnection for the Gross--Pitaevskii equation
(Theorem~\ref{T.GP}). For convenience, we will divide the proof in
three steps:

\subsubsection*{Step 1: Construction of a local solution using
  noncharacteristic hypersurfaces in spacetime}

Let $\Si\subset\RR^4$ be a pseudo-Seifert surface connecting the
curves~$\Ga_0,\Ga_1\subset\RR^3$ in time~$T$, as defined in the
Introduction. The existence of these surfaces is standard because all
closed curves in~$\RR^4$ are isotopic. (It should be noticed,
however, that in general one cannot choose~$\Si$ as a knot cobordism,
that is, homeomorphic to $\Ga_0\times(0,1)$.) One can also
extend~$\Si$ so that it is defined for times slightly smaller than~0
and slightly larger than~$T$.

Moreover, if necessary one can deform the
curves $\Ga_0,\Ga_1$ with a smooth diffeomorphism arbitrarily close to
the identity in the $C^k$~norm to ensure that the curves $\Ga_0,\Ga_1$
and the surface~$\Si$ are real analytic, and that~$\Si$ is in general
position with respect to the time axis in the sense that
the set of points
\[
\{(x,t)\in\Si: e_4\in N_{(x,t)}\Si\}
\]
is finite. Here $e_4:= (0,0,0,1)$ is the time direction and
$N_{(x,t)}\Si$ is the normal plane of~$\Si$ at the
point~$(x,t)$. We will label all the points of this form as
\begin{equation}\label{badpts}
\{X_j:=(x_j,t_j)\}_{j=1}^J\,.
\end{equation}
Equivalently, this means that, after deforming~$\Si$ by
a small diffeomorphism if necessary, the coordinate~$t$ is a Morse
function on~$\Si$, so the claim is a straightforward consequence of the
density of Morse functions~\cite[Theorem 6.1.2]{Hirsch}.
It is also standard that we can also assume all the critical points
$X_j$ of the
function $t|_\Si$ correspond to different critical values, meaning
that $t_k\neq t_j$ for all $1\leq j\neq k\leq J$. For future
reference, let us denote the set of these critical times by
\begin{equation}\label{cP}
\cP:=\{ t_j : 1\leq j\leq J\}\,.
\end{equation}
Observe that an equivalent characterization of this set is
\[
\cP:=\{t\in (0,T_1): e_4\in N_{(x,t)}\Si \text{ for some point }
(x,t)\in\Si\}\,.
\]

We now claim that there exists a vector field $a(x,t)$ on~$\RR^4$ such
that, for every point $(x,t)\in\Si$, the vector
\begin{equation}\label{defW}
W(x,t):= a(x,t)\wedge \tau_1\wedge \tau_2
\end{equation}
is not parallel to the time direction, $e_4$ (in particular,
nonzero). Here $(\tau_1,\tau_2)$ is any oriented orthonormal basis of
the tangent space $T_{(x,t)}\Si$ and the product of three vectors
$V_1\wedge V_2\wedge V_3$ in~$\RR^4$ is defined as~0 if the vectors
are linearly dependent and as the only vector, modulo a
multiplicative factor that is inessential for our present purpose, that is orthogonal to $V_1$, $V_2$
and~$V_3$. (The multiplicative factor is of course determined by the
norms of~$V_j$ and the orientation of~$\RR^4$.)

To show the existence of the vector field~$a(x,t)$, let us recall that the
normal bundle of~$\Si$ in~$\RR^4$ is trivial~\cite{Massey}, so there are analytic
vector fields $N_1(x,t)$ and~$N_2(x,t)$ on~$\RR^4$ such that
\[
N_{(x,t)}\Si =\Span\{ N_1(x,t), N_2(x,t)\}\,.
\]
One can then write the vector field~$a$ as
\[
a(x,t)= f_1(x,t)\, N_1(x,t)+ f_2(x,t)\, N_2(x,t)\,,
\]
where $f_1,f_2$ are analytic real-valued functions on~$\RR^4$ to be
determined. Since obviously the vector~$W(x,t)$ must be normal to the
surface~$\Si$ at the point~$(x,t)$ and~$e_4$ is only normal at the
finite number of points specified in~\eqref{badpts}, it is clear
that the only conditions that the functions~$f_1,f_2$ must satisfy are
\[
[f_1(x_j,t_j)\, N_1(x_j,t_j) + f_2(x_j,t_j)\, N_2(x_j,t_j)]\cdot e_4\neq0
\]
for $1\leq j\leq J$ to ensure that $W(x_j,t_j)$ is not parallel
to~$e_4$ at these points and
\[
f_1(x,t)^2 + f_2(x,t)^2\neq0 
\]
for all $(x,t)\in\Si$ to make sure that $W(x,t)$ is always
nonzero. The existence of the functions $f_1,f_2$ is then apparent.

Now that we have the vector field~$a(x,t)$, we can next define an analytic
hypersurface~$S$ in~$\RR^4$ as
\[
S:=\{(x,t)+s\, a(x,t): (x,t)\in\Si,\; |s|<s_0\}\,,
\]
where $s_0>0$ is small enough. It follows from the construction that
the normal direction~$N$ of~$S$, which is proportional to the vector
field~$W$ defined in~\eqref{defW} modulo an error of size $O(s_0)$, is never parallel to the time
direction, $e_4$. Furthermore, it is clear from the definition of~$S$
that it is diffeomorphic to the product $\Si\times (-s_0,s_0)$, so one
can take a real-valued analytic function $\phi:S\to\RR$ on the
hypersurface~$S$ such that
\begin{equation}\label{phiSi}
  \phi^{-1}(1)=\Si
\end{equation}
and its gradient is
transverse to~$\Si$, i.e., that its intrinsic gradient (or covariant
derivative) $\nabla_S\phi$ does not vanish on~$\Si$.

Consider the Cauchy problem
\begin{equation}\label{CK}
i\pd_t v+\De v=0\,,\qquad v|_S=\phi\,,\qquad N\cdot \nabla_{x,t} v=i\,,
\end{equation}
where $N$ is a unit normal vector to the hypersurface~$S$ in~$\RR^4$
and $\nabla_{x,t} v$ denotes the spacetime gradient of~$v$. The
hypersurface~$S$ is non-characteristic for the Schr\"odinger equation
because $N$ is never parallel to the time direction. Hence the Cauchy--Kowalewskaya  theorem
ensures that there exists a real analytic solution~$v$ to the problem~\eqref{CK}
defined in a neighborhood~$V\subset\RR^4$ of~$S$.

An important observation is that
\begin{equation}\label{v1}
  v^{-1}(1)=\Si
\end{equation}
provided that we take a
small enough neighborhood~$V$. In order to see this, let us denote the real and imaginary parts of~$v$ by
\[
v=v_1+iv_2
\]
and notice that the Cauchy conditions we have imposed can be rewritten as
\[
v_1|_S=\phi\,, \quad  v_2|_S =0\,,\quad N\cdot \nabla_{x,t}
v_1=0\,,\quad  N\cdot \nabla_{x,t} v_2=1\,.
\]
Therefore $v_2$ only vanishes on~$S$, while $(v_1|_S)^{-1}(1)=\Si$. Furthermore,
the gradients of~$v_1$ and~$v_2$ are transverse
on~$\Si=v_1^{-1}(1)\cap v_2^{-1}(0)$, that is,
\begin{equation*}
\rank (\nabla_{x,t} v_1,\nabla_{x,t} v_2)=2\quad \text{on } \Si\,.
\end{equation*}
This is clear because $N\cdot \nabla_{x,t}v_1|_S=0$, which ensures that
\[
\nabla_{x,t}v_1|_S=\nabla_S\phi\,,
\]
which is transverse to~$\Si$ by the definition of~$\phi$, while
$\nabla_{x,t}v_2|_S=N$. 

\subsubsection*{Step~2: Robust geometric properties of the local solution}

By Equation~\eqref{v1}, it is clear that
\[
Z_{1-v}(t)=\Si_t
\]
for all~$t$, where we recall that $\Si_{t_0}:=\{(x,t)\in \Si: t=t_0\}$ is the
intersection of~$\Si$ with the time~$t_0$ slice.  As~$t$ is a Morse
function on~$\Si$ by construction, the reconnection times~$t_j$ must
be critical values of the function~$t|_\Si$ because the
constant time slice and the surface~$\Si$ stop being transverse (that
is, the vector $e_4$ belongs to the normal plane at the critical
point $X_j:=(x_j,t_j)\in\Si$). Besides, as the critical level $\{t=
t_j\}\cap \Si$ must be a curve (i.e., of dimension~1), it follows
that the critical point $X_j$ must be a saddle point, that is, of
Morse index~1. The Morse lemma then ensures that there are smooth
local coordinates $(y_1,y_2)$ on a small neighborhood of~$X_j$
in~$\Si$ such that, in that neighborhood,
\begin{equation}\label{hyperb}
t|_\Si= t_j+ y_1^2-y_2^2\,.
\end{equation}
Furthermore, with $X:=(x,t)\in\RR^4$, there are two linearly independent vectors $V_1,V_2$ on~$\RR^4$
such that
\begin{equation}\label{defy}
y_j=V_j\cdot (X-X_j)+ O(|X-X_j|^2)\,.
\end{equation}
Since the distance between the two sheets of the
hyperbola~\eqref{hyperb} is
\[
2|t-t_j|^{1/2}
\]
when measured with respect to the metric $dy_1^2+dy_2^2$, it follows
from~\eqref{defy} that the distance~$d_j(t)$ between the corresponding two
components of the set $Z_{1-v}(t)$ near the reconnection time~$t_j$ is
bounded as
\begin{equation}\label{dk}
\frac1C |t-t_j|^{1/2}\leq d_j(t)\leq C |t-t_j|^{1/2}\,.
\end{equation}
It is also a standard consequence of Morse theory for functions on a
surface whose critical points have all distinct critical values
that the parity of the number of components of a level set changes as
one crosses a critical value.

The above geometric construction is robust under suitable
perturbations of the function~$v$. More precisely, Thom's transversality theorem~\cite[Theorem
20.2]{AR} ensures that, given any~$k\geq1$ and~$\de>0$, there exists some~$\ep>0$ such that:
\begin{enumerate}
\item The level set of value~1 of any function~$w$ with
\begin{equation}\label{Thom1}
\|w-v\|_{C^k(V)}<\de
\end{equation}
satisfies
\begin{equation}\label{Thom2}
w^{-1}(1)\cap V= \Psi(\Si)\,,
\end{equation}
where $\Psi$ is a smooth diffeomorphism of~$\RR^4$ with
$\|\Psi-\id\|_{C^k(\RR^{n+1})}<\ep$.
\item There is a finite union of closed intervals $\cI\subset (0,T)$ containing
  the set~$\cP$ (cf.~\eqref{cP}) of total length less than~$\ep$ and a
  continuous one-parameter family of diffeomorphisms
  $\{\Phi^t\}_{t\in\RR}$ of~$\RR^3$ with
  $\sup_{t\in\RR}\|\Phi^t-\id\|_{C^k(\RR^3)}<\ep$ such that
  $\Phi^t(\Si_t)= Z_{1-w}(t)\cap V$ for all
  $t\in [0,T]\backslash\cI$.

  \item The distance~$d_j(t)$ between the two
components of the set $Z_{1-w}(t)\cap V$ near a critical point~$X_j'$
of index~1 of the
Morse function $t|_{\Phi(\Si)}$ is
bounded as in Equation~\eqref{dk}.

\item The parity of the number of components of $Z_{1-w}(t)\cap V$
  is different at each time $t= t_j-\de'$ and $t_j+\de'$, for any 
  small enough $\de'>0$. 
\end{enumerate}

\begin{remark}\label{R.cI}
By an easy transversality argument, one can provide a more exhaustive
description of the zero set $Z_{1-w}(t)$ as follows. For $t\in
[0,T]\backslash\cI$, item~(ii) means that  $Z_{1-w}(t)\cap V$ is a small
deformation of the smooth embedded curve $\Si_t$ that corresponds
to the intersection of the spacetime surface~$\Si$ with the time slice
$\RR^3\times\{t\}$. The times~$t_0\in \cP$ are those at which
the curve $\Si_{t_0}$ self-intersects, so that $\Si_{t_0}$ is then a smooth
immersed curve. What happens is that there is another finite
set~$\cP'$, at a distance at most~$\ep$ of~$\cP$ and of the same
cardinality, such that for all $t_0'\in\cP'$, $Z_{1-w}(t_0')\cap V$ is also a
smooth immersed curve, while for $t\in \cI$ slightly above or below a
critical time $t_0'\in\cP'$ as above, the zero set
$Z_{1-w}(t)\cap V$ is a smooth embedded curve with the same structure as
before.
\end{remark}

\subsubsection*{Step 3: Construction of the solution to the
  Gross--Pitaevskii equation}

Theorem~\ref{T.GAT} and Remark~\ref{R.reg} guarantee that there
exists a Schwartz function $w_0\in
\cS(\RR^n)$ such that $w:=e^{it\De}w_0$ approximates the above
function~$v$ as
\[
\|w-v\|_{C^k(V)}<\ep/2\,,
\]
where ~$k\geq1$ can be chosen at will.

Let us now consider the rescaled Gross--Pitaevskii equation
\[
i\pd_t\tu+\De \tu+\de(1-|\tu|^2) \tu=0
\]
on~$\RR^3$ with initial datum
\[
\tu(x,0)=1-w_0(x)\,,
\]
where~$\de>0$ is a small constant. In view of Duhamel's formula
\[
\tu(x,t)= 1-w(x,t)+ i\de\int_0^t e^{i(t-s)\De}(1-|\tu(x,s)|^2)\tu(x,s)\, ds\,,
\]
it is standard (see
e.g.~\cite{Tao}) that, 
for all small enough~$\de$, there exists a global
solution~$\tu$ to this
equation with
\[
1-\tu\in C^\infty\loc(\RR,\cS(\RR^3))\,,
\]
which is bounded as
\[
\|\tu-1+ w\|_{C^k([-T,T]\times\RR^3)} \leq C_T\de
\]
for any $T>0$. The constant~$C$ depends on~$T$ and~$w_0$ but not on~$\de$. It then follows from our application of Thom's isotopy
theorem~\eqref{Thom1} in Step~2 that the zero set $\tu^{-1}(0)$
satisfies~\eqref{Thom2} for some smooth diffeomorphism~$\Psi$
of~$\RR^4$ with $\|\Psi-\id\|_{C^k(\RR^4)}<\ep$, and that the zero set
$Z_{\tu}(t)$ is of the form described in item~(ii) of
Step~2. Also by Step~2, $\tu$~satisfies the $t^{1/2}$~law and the change of
parity property.

Notice that the function
\[
u(x,t):=  \tu (\de^{-1/2}x, t/\de)
\]
satisfies the Gross--Pitaevskii equation
\[
i\pd_tu+\De u+(1- |u|^2) u=0\,,\qquad u(x,0)=1- w_0(\de^{-1/2}x)
\]
and tends to~1 as
\[
1-u\in C^\infty\loc(\RR,\cS(\RR^3))\,.
\]
Since $u$ is just an (anisotropic) rescaling of~$\tu$, we infer that
the zero set of~$u$ is of the form described in the statement of the theorem.

\section{Comparison with experimental observations and solutions
  with other conditions at infinity}
\label{S.observed}

\subsection{Comparison with experimental results}

A remarkable feature of the strategy that we have employed to prove
the existence of solutions to the Gross--Pitaevskii equation featuring
vortex reconnection is that it presents the same qualitative
properties that are observed in the physics literature:

\subsubsection*{The $t^{1/2}$ law}

As we discussed in the Introduction, in the reconnection scenarios
that we construct, the distance between reconnecting vortices near the
reconnection time~$T$ behaves as $|t-T|^{1/2}$. This is in perfect
agreement with the~$t^{1/2}$ law for the separation velocities
that have measured in the
laboratory~\cite{Bewley}, observed numerically~\cite{Villois} and
heuristically explained in~\cite{Nazarenko}.

\subsubsection*{Change of parity of the components at reconnection}

We also stressed that the parity of the number of reconnecting quantum
vortices is numerically observed to change at each reconnection time~\cite{Irvine}, as depicted
in Figure~\ref{Figure}. The reconnection scenarios
we construct also feature this property as an indirect consequence of
Morse theory for functions on surfaces.

\subsubsection*{Birth and death of quantum vortices}

Numerical simulations also show that quantum vortices can be created
or destroyed~\cite{Irvine}, as also shown in Figure~\ref{Figure}
(bottom). This is a degenerate case of vortex reconnection that
appears, in our scenarios, whenever the Morse function $t|_\Si$ has any
local extrema.

\subsubsection*{Pseudo-Seifert surfaces as a universal scenario
  of vortex reconnection}

The time evolution of a quantum vortex (which is generically, at each
time, a smooth curve in~$\RR^3$) automatically defines a surface~$\Si$
is spacetime~$\RR^4$. Generically, this surface is smooth by Sard's
theorem and the time coordinate is a Morse function
on~$\Si$. Therefore, the description of vortex reconnection we use in
the construction of the scenarios is, in a way, universal.

\subsection{Solutions to NLS that decay at infinity: the
  case of laser beams}

It is folk wisdom in physics that, in the Gross--Pitaevskii equation,
if one replaces the asymptotic condition $u(x,t)\to 1$ as
$|x|\to\infty$ by a decay condition (e.g., that~$u$ be square
integrable), it should be easier to show that there is a wealth of
reconnections. In the language of physics, this is because the
condition $u\to 1$ is associated with the existence of a chemical
potential at infinity. In constrast, the decay condition $u\to0$ corresponds to the
more flexible case of optical vortices, which describes laser beams~\cite{Dennis}.

We shall next mention how the strategy that we have
developed applies to the case of laser beams (and to many other
nonlinear Schr\"odinger equations). Remarkably, we do find
that in this setting the argument leads to a stronger reconnection
theorem, in that the diffeomorphism that appears in the statement can
be arbitrarily close to the identity:

\begin{theorem}\label{T.laser}
  Consider two links $\Ga_0,\Ga_1\subset\RR^3$ and a pseudo-Seifert
  surface~$\Si\subset\RR^4$
  connecting~$\Ga_0$ and~$\Ga_1$ in time~$T>0$. For any $\ep>0$ and
  any~$k>0$, there is a Schwartz initial datum $u_0\in\cS(\RR^3)$ such
  that the corresponding solution to the Gross--Pitaevskii equation $u\in C^\infty\loc(\RR, \cS(\RR^3))$
  realizes the vortex reconnection pattern
  described by~$\Si$ up to a small deformation. More precisely, for
  any fixed $\ep>0$ and $k>0$, the
  properties (ii)--(iv) of Theorem~\ref{T.GP} hold with $\eta:=1$.
\end{theorem}

\begin{proof}
The proof goes just as in Theorem~\ref{T.GP}. Indeed, Steps~1 and~2
apply directly in this setting, the only difference being that the
level sets $v^{-1}(1)$ and $Z_{1-v}(t)$ (and similarly for~$w$) have
to be replaced by $v^{-1}(0)$ and $Z_v(t)$, and that the
condition $\phi^{-1}(1)=\Si$ (Equation~\eqref{phiSi}) must be replaced by
\[
\phi^{-1}(0)=\Si\,.
\]

In Step~3, one similarly considers the rescaled modified Gross--Pitaevskii equation
\[
i\pd_t\tu+\De \tu-\de|\tu|^2 \tu=0
\]
on~$\RR^3$ with initial datum
\[
\tu(x,0)=w_0(x)\,,
\]
where~$\de>0$ is a small constant. Another easy argument using
Duhamel's formula then yields that for all small enough~$\de$ there
exists a global solution $\tu\in C^\infty\loc(\RR,\cS(\RR^3))$ to this
equation, which is bounded as
\[
\|\tu- w\|_{C^k([-T,T]\times\RR^3)} \leq C\de\,.
\]
The results about the robustness of the geometric properties of~$v$
proved in Step~2 obviously apply to~$\tu$. If we now note that the function
\begin{equation}\label{eqfinal}
u(x,t):= \de^{1/2} e^{it}\,\tu (x,t)
\end{equation}
satisfies the Gross--Pitaevskii equation
\[
i\pd_tu+\De u+ (1- |u|^2) u=0
\]
and has the same zero set as~$\tu$, the result then follows. 
\end{proof}

\begin{remark}\label{R.Schrodinger}
It is clear from the proof that Theorem~\ref{T.laser} holds verbatim if one replaces the Gross--Pitaevskii
equation by a NLS equation of the form
\[
i\pd_t u+\De u + V(u,\overline u)=0
\]
provided that the nonlinearity $V(u,\overline u)$ is subcritical, a
smooth enough function of~$u$ and~$\overline u$, and of order $o(|u|)$
for $|u|\ll 1$. In particular, the result obviously holds for the
linear Schr\"odinger equation.
\end{remark}

\subsection{The periodic case: the Gross--Pitaevskii equation
  on~$\TT^3$}

To conclude, we shall next sketch how the above results can be
extended to the case of the Gross--Pitaevskii equation
\[
i\pd_t u+\De u + (1-|u|^2) u=0\,,\qquad u(x,0)=u_0(x)\,,
\]
when the spatial variable takes values in the 3-torus $\TT^3:=
(\RR/2\pi\ZZ)^3$.

For the ease of notation, we regard the unit ball $B_1\subset\RR^3$
as a subset of~$\TT^3$ with the obvious identification. The dilation
$\La_\eta$, introduced in~\eqref{dilation}, can then be understood as
a map from $B_1\subset\TT^3$ into itself provided that $\eta<1$.

\begin{theorem}\label{T.T3}
  Consider two links $\Ga_0,\Ga_1$ contained in the unit ball
  $B_1\subset\RR^3$ and a pseudo-Seifert surface~$\Si$
  connecting~$\Ga_0$ and~$\Ga_1$ in time~$T>0$. We assume that $\Si$
  is contained in the spacetime cylinder $B_1\times\RR$. For any
  $\ep>0$ and any~$k>0$, there is an initial datum
  $u_0\in C^\infty(\TT^3)$ such that the corresponding solution to the
  Gross--Pitaevskii equation $u\in C^\infty\loc(\TT^3\times\RR)$
  realizes the vortex reconnection pattern described by~$\Si$ up to a
  diffeomorphism. Specifically:
  \begin{enumerate}
  \item The evolution of the vortex set  $Z_u(t)$ is known for all
    times during the reconnection process: there is some $\eta>0$ and
  a diffeomorphism $\Psi$ of~$\TT^3\times\RR$ with
  $\|\Psi-\id\|_{C^k(\TT^3\times\RR)}<\ep$ such that $\La_{\eta}[\Psi(\Si)_t]$ is a union
  of connected components of $Z_u(\eta^2 t)$ for all~$t\in [0,T]$.
  
\item In particular, there is a smooth one-parameter family of
  diffeomorphisms $\{\Phi^t\}_{t\in\RR}$ of~$\TT^3$ with
  $\|\Phi^t-\id\|_{C^k(\TT^3)}<\ep$ and a finite
  union of closed
  intervals~$\cI\subset (0,T)$ of total length less than~$\ep$ such that $\La_\eta[\Phi^t(\Si_t)]$
  is a union of connected components of the set $Z_u(\eta^2 t)$ for all $t\in
  [0,T]\backslash\cI$.
  
\item The separation distance obeys the $t^{1/2}$ law and the parity
  of the number of quantum vortices of~$\Phi^t(\Si_t)$ changes at each
  reconnection time.
  \end{enumerate}
  
\end{theorem}

\begin{proof}
  By Remark~\ref{R.Schrodinger}, there is a Schwartz initial datum
  $v_0\in\cS(\RR^3)$ such that the solution to the Schrödinger
  equation on~$\RR^3$ $v:=e^{it\De}v_0$ realizes the reconnection
  pattern defined by~$\Si$ up to a small deformation and satisfies the
  properties described in Theorem~\ref{T.laser}. We can assume that
  this reconnection takes place in the bounded spacetime domain
  $B_1\times (0,T)$.
  
  Note that~$v$ can be
  written in terms of the Fourier transform of~$v_0$ as
  \[
v(x,t)=\int_{\RR^3} e^{i\xi\cdot x -i|\xi|^2 t}\, \hv_0(\xi)\, d\xi\,.
  \]
As $\hv_0$ is a Schwartz function, it is standard that, for $(x,t)$ in
the bounded set $B_1\times (0,T)$, the above integral can be
approximated by a Riemann sum of the form
\[
v_1(x,t):= J^{-6}\sum_{j=1}^{J^6} e^{i\xi_j\cdot x -i|\xi_j|^2 t}\, \hv_0(\xi_j)
\]
as
\begin{equation}\label{vv1}
\|v-v_1\|_{C^k(B_1\times(0,T))}<\ep\,,
\end{equation}
where $k$ and $\ep>0$ are fixed but arbitrary and~$J$ is a large
positive integer. One possible way of
choosing the points~$\xi_j$ is by taking a
cube of side $J$ centered at the origin, dividing it into $J^6$
cubes of side $J^{-1}$ and letting $\xi_j$ be any point in the
$j^{\mathrm{th}}$ cube. For a large enough~$J$, it is clear that the
approximation bound~\eqref{vv1} will hold. It is also apparent that
one can pick all the points $\xi_j$ rational, i.e., $\xi_j\in\mathbb
Q^3$. Observe that the approximation estimate and the stability under
small perturbations of reconnection scenarios that we constructed in
Theorem~\ref{T.laser} ensures that $v_1$ features reconnections that
are diffeomorphic to, and a small deformation of, those described by
the pseudo-Seifert surface~$\Si$.

Let $N$ be the height of the point $(\xi_1,\xi_2,\dots, \xi_{J^6})\in
\mathbb Q^{3J^6}$, that is, the least common denominator of its
coordinates in reduced form, and define
\[
w(x,t):= v_1(Nx,N^2 t)\,.
\]
By the way we have picked~$N$ it is clear that $w(x,t)$ defines a
function in $C^\infty\loc(\TT^3\times\RR)$ that satisfies the
Schr\"odinger equation on the 3-torus:
\[
i\pd_t w+\De w=0\,.
\]
The zero set of~$w$ is simply the image of the zero set
of~$v_1|_{(-\pi N,\pi N)^3\times \RR}$ under the
map
\[
\Theta_N(x,t):= (x/N,t/N^2)\,.
\]
In view of the properties of~$v_1= w\circ \Theta_N$, this immediately implies that $w$ features reconnections contained in
the set $B_{1/N}\times (0,T/N^2)$ that are diffeomorphic to the
scenario described by~$\Si$ and satisfy the properties of the
statement. This fact is robust under small perturbations.

It is now easy to promote this solution of the linear Schr\"odinger
equation on~$\TT^3$ to a global smooth solution of the modified
Gross--Pitaevskii equation
\[
i\pd_t\tu+\De \tu-\de|\tu|^2 \tu=0
\]
on~$\TT^3$ that is close to~$w$ using Duhamel's formula and Bourgain's
dispersive estimates on the torus~\cite{Bourgain}. This can be
transformed into a solution to the Gross--Pitaevskii equation using
the formula~\eqref{eqfinal} just as in Theorem~\ref{T.laser}. This
completes the proof of the theorem with $\eta:=1/N$.
\end{proof}

\section*{Acknowledgements}

The authors are indebted to Luis Escauriaza for explanations
concerning three-sphere inequalities for the Schr\"odinger equation,
including the counterexample that we can construct using a Tychonov-type
argument. The authors are supported by the ERC Starting Grants~633152
(A.E.) and 335079 (D.P.-S.) and by the grant MTM-2016-76702-P of the Spanish Ministry of Science
(D.P.-S.). This work is supported in part by the
ICMAT--Severo Ochoa grant
SEV-2015-0554.

\appendix


\section{Uniform lower bounds for the integral of a Bessel function}
\label{A.Bessel}

Notice that the function $\cI_\nu(\al)$, introduced in~\eqref{defcI},
can be written as
\[
  \cI_\nu(\al):= \int_0^{R''} r |I_\nu(\al r)|^2\, dr =
  \frac1{|\al|^2}\int_0^{|\al| {R''}} \rho |Z_\nu(\rho)|^2\, d\rho\,,
\]
where we define $Z_\nu(z):= J_\nu(z)$ if $\al\in i\RR^+$ and
$Z_\nu(z):= I_\nu(z)$ if $\al\in \RR^+$. Hence a first observation that is useful when computing lower bounds for
the function~$\cI_\nu(\al)$ is that $|\al|^2 \cI_\nu(\al)$ is an
increasing function of~$|\al|$.

We will write
$F\approx G$ if there is a positive constant (which does not depend on
$\nu$ or~$\al$), such that
\[
\frac F C\leq G\leq CF\,.
\]
Likewise, $F\lesssim G$ means that $F\leq CG$ with $C$ as above,
and $F\ll G$ means that $F\leq \de G$ for a certain small constant~$\de$,
again independent of~$\nu$ or~$\al$.

\begin{lemma}\label{L.Bessel}
With $\al\in\RR^+\cup i\RR^+$ and $\nu\geq\frac12$, the function
$\cI_\nu(\al)$ satisfies the lower bound
\[
\cI_\nu(\al) \gtrsim \frac1{\langle|\al|\rangle^2 \nu^2}\Big(\frac
{C\min\{|\al|,1\}}\nu\Big)^{2\nu} e^{C \Real\al} \,.
\]
\end{lemma}

\begin{proof}
  We need to analyze the different cases separately:\smallskip

  \step{Case 1: $|\al|\lesssim 1$ and $\nu\lesssim1$} The asymptotics
  for Bessel functions near zero~\cite[8.440 and 8.445]{GR}, $I_\nu
  (z)= C z^\nu +O(z^{\nu+1})$, together
  with the fact that $I_\nu(z)$ is obviously of order~1 if $\nu$
  and~$|z|$ are, immediately yield that
  \[
\cI_\nu(\al)\approx |\al|^{2\nu}
  \]
in this region of the parameter space.

\step{Case 2: $|\al|\gg 1$ and $\nu\lesssim1$} The usual large time
asymptotics for Bessel functions of fixed order~\cite[8.451.1 and 8.451.5]{GR},
\[
  I_\nu(z)= \begin{cases}
    C z^{-1/2}{e^z}[1+ O(z^{-1})] & \text{if } z\in\RR^+\,,\\
    C |z|^{-1/2}{\cos(|z|-c_\nu)}+ O(|z|^{-3/2}) & \text{if } z\in i\RR^+\,,
  \end{cases}
\]
ensure that
\[
\cI_\nu(\al)\approx \int_0^{R''} \frac{e^{2\al r}}{\al}\, dr\approx
\frac{e^{2{R''} \al}}{\al^2}
\]
if $\al\in\RR^+$, while for $\al\in i\RR^+$ 
\[
\cI_\nu(\al)= \frac1{|\al|^2}\int_0^{ |\al| {R''}} \rho\, |I_\nu(\rho)|^2\,
d\rho\approx \frac1{|\al|^2}\int_0^{|\al| {R''}} {\cos^2(\rho - \rho_0)}\,
d\rho \approx \frac 1{|\al|}\,.
\]

\step{Case 3: $|\al|\lesssim1$ and $\nu\gg1$} The asymptotic
expansions for Bessel functions of large order~\cite[8.452.1]{GR},
\[
I_\nu(z) =C_1 \nu^{-\frac 12}\Big(\frac{C_2z}\nu\Big)^{\nu}[1+ O(\nu^{-1})]\,,
\]
ensure
that
\[
\cI_\nu(\al)\approx \frac1\nu\int_0^{R''} r\Big|\frac{C\al
  r}\nu\Big|^{2\nu}\, dr\approx \frac{(C |\al|)^{2\nu}}{\nu^{2\nu+2}}\,.
\]

  \step{Case 4: $|\al|\gg1$ and $\nu\gg 1$} Since $|\al|^2\,
  \cI_\nu(\al)$ is an increasing function of~$|\al|$, from Case~3 we
  immediately get that
  \begin{equation}\label{lowbound}
\cI_\nu(\al)\gtrsim \frac{ C^{2\nu} }{|\al|^2 \nu^{2\nu+2}}\,.
\end{equation}

Putting all the cases together we arrive at the bounds in the statement.
\end{proof}

\bibliographystyle{amsplain}

\end{document}